\documentclass[10pt, reqno]{amsart}
\numberwithin{equation}{section}
\usepackage{amssymb}
\usepackage{hyperref}
\usepackage{cite}
\usepackage{caption, color}
\usepackage{subcaption}
\usepackage[normalem]{ulem}
\usepackage{float}

\usepackage{tikz, pgfplots}
\usetikzlibrary{shapes,snakes}
\usetikzlibrary{patterns}
\usepackage{multirow}

\allowdisplaybreaks

\newcommand{\qtq}[1]{\quad\text{#1}\quad}

\let\Re=\undefined\DeclareMathOperator*{\Re}{Re}

\newcommand{\R}{\mathbb{R}}
\newcommand{\C}{\mathbb{C}}

\newcommand{\eps}{\varepsilon}

\newcommand{\F}{\mathcal{F}}

\newtheorem{theorem}{Theorem}[section]

\newtheorem{lemma}[theorem]{Lemma}

\newtheorem{conjecture}[theorem]{Conjecture}
\newtheorem{proposition}[theorem]{Proposition}

\theoremstyle{definition}

\theoremstyle{remark}

\begin{document}

\title[Nonlinear wave equations]{Numerical simulations for the energy-supercritical nonlinear wave equation}

\author{Jason Murphy}
\address{Department of Mathematics and Statistics, Missouri University of Science and Technology, Rolla, MO, USA}
\email{jason.murphy@mst.edu}

\author{Yanzhi Zhang}
\address{Department of Mathematics and Statistics, Missouri University of Science and Technology, Rolla, MO, USA}
\email{zhangyanz@mst.edu}

\maketitle

\begin{abstract} We carry out numerical simulations of the defocusing energy-supercritical nonlinear wave equation for a range of spherically-symmetric initial conditions.  We demonstrate numerically that the critical Sobolev norm of solutions remains bounded in time.  This lends support to conditional scattering results that have been recently established for nonlinear wave equations.
\end{abstract}

\section{Introduction}

In recent years, there has been a great deal of progress in understanding the long-time behavior of solutions to nonlinear dispersive partial differential equations.  One line of research has focused on the scattering problem for large solutions under optimal regularity assumptions on the initial conditions, particularly in the setting of defocusing nonlinear Schr\"odinger equations (NLS) and wave equations (NLW).  Progress in this direction was precipitated especially by the development of new techniques (e.g. the concentration compactness approach to induction on energy) that were developed in order to establish global well-posedness and scattering (i.e. asymptotically linear behavior) for certain special cases, e.g. the mass- and energy-critical NLS.  Outside of these special cases, however, current techniques are often limited to proving conditional results, in which one shows that scattering occurs under the assumption of \emph{a priori} bounds for a critically-scaling Sobolev norm.  In this paper, we will present numerical simulations for the energy-supercritical NLW that lend support to the veracity of these assumed critical bounds.  A similar study was carried out in \cite{2659737} in the setting of the energy-supercritical NLS.  A related study also appeared in \cite{StraussVazquez} in the setting of the nonlinear Klein--Gordon equation.  See also \cite{DonSch} for a numerical study of the boundary of the forward scattering region for the \emph{focusing} nonlinear Klein--Gordon equation. 

To describe the problem and our results more precisely, we introduce the equations
\begin{equation}\label{nls}\tag{NLS}
i\partial_t u + \Delta u = \mu|u|^p u,\quad u:\R_t\times\R_x^d\to\C
\end{equation}
and
\begin{equation}\label{nlw}\tag{NLW}
-\partial_t^2 u + \Delta u = \mu|u|^p u,\quad u:\R_t\times\R_x^d\to\R.
\end{equation}
In each case, the parameter $\mu$ yields either the defocusing {($\mu > 0$)} or focusing  {($\mu < 0$)} case, and $p>0$ is the power of the nonlinearity.  These are Hamiltonian PDE, with the conserved energy given by
\[
E(u) = \int_{\R^d} \tfrac12 |\nabla u|^2 + \tfrac{\mu}{p+2}|u|^{p+2}\,dx\qtq{for NLS}
\]
and 
\begin{equation}\label{energy}
E(u,\partial_t u) = \int_{\R^d} \tfrac12 |\partial_t u|^2 + \tfrac12 |\nabla u|^2 + \tfrac{\mu}{p+2}|u|^{p+2}\,dx\qtq{for NLW.}
\end{equation}

Both equations also enjoy a scaling symmetry, namely
\begin{equation}\label{scaling}
u(t,x)\mapsto \begin{cases} \lambda^{\frac{2}{p}}u(\lambda^2 t,\lambda x)  & \text{for NLS} \\ \lambda^{\frac{2}{p}} u(\lambda t,\lambda x) & \text{for NLW,}\end{cases} 
\end{equation}
which defines a notion of \emph{critical regularity} for these equations.  In particular, if we define
\begin{equation}\label{sc}
s_c = \tfrac{d}{2}-\tfrac{2}{p},
\end{equation}
then one finds that the  $\dot H^{s_c}$-norm of $u|_{t=0}$ for NLS and the $\dot H^{s_c}\times \dot H^{s_c-1}$ norm of $(u,\partial_t u)|_{t=0}$  for NLW are invariant under the rescaling \eqref{scaling}.\footnote{Here $\dot H^s$ denotes the homogeneous $L^2$-based Sobolev space; see Section~\ref{S:Notation}.} Generally speaking, these are the optimal spaces for initial data in terms of the well-posedness theory of \eqref{nls} and \eqref{nlw}; see e.g. \cite{Cazenave, Lindblad-Sogge, CCT}.  

The main topic of this paper is the question of {scattering}.  We say that a forward-global solution $u$ to \eqref{nlw} \emph{scatters} (in $\dot H^{s_c}\times\dot H^{s_c-1}$) if there exists a solution $v(t)$ to the \emph{linear} wave equation such that
\[
\lim_{t\to\infty} \|(u(t),\partial_t u(t))-(v(t),\partial_t v(t))\|_{\dot H^{s_c}\times \dot H^{s_c-1}} = 0. 
\]
An analogous definition holds for solutions to \eqref{nls}. 

A special case of \eqref{nls} and \eqref{nlw}, called the \emph{energy-critical} case, occurs when the scaling symmetry \eqref{scaling} leaves the energy of the solution invariant as well.  This corresponds to choosing $p=\tfrac{4}{d-2}$ in dimensions $d\geq 3$, or equivalently $s_c=1$.  For the case of NLS, there is also the \emph{mass-critical} case corresponding to $p=\tfrac{4}{d}$, in which case the scaling symmetry preserves the mass (i.e. the $L^2$ norm), which is a conserved quantity for NLS (but not for NLW). 
For these special cases, conservation of energy/mass yields \emph{a priori} control over the critical Sobolev norm (in the defocusing case, at least).  Ultimately, this provides enough control over solutions to establish global well-posedness and scattering, although proving this is a very challenging problem that required the work of many mathematicians over many years to settle definitively (see \cite{Bou1, CKSTT, Dod1, Dod2, Dod3, Dod4, Dod5, Gri, KenMer, KTV, KV2, KV3, KVZ, RV, Tao1, TVZ, Vis0, Vis1, Vis2, BahGer, Grillakis, Grillakis2, Kapitanski, KenMer0, LiZha, Nakanishi, ShaStr, Struwe}):

\begin{theorem}[Global well-posedness and scattering]\label{T:scatter} For the defocusing case of the mass- and energy-critical NLS or energy-critical NLW, arbitrary initial data in the critical Sobolev space lead to global solutions that {scatter}. Similar results hold in the focusing case, provided one imposes suitable size restrictions on the mass/energy.
\end{theorem}

The resolution of Theorem~\ref{T:scatter} required the development of a powerful new set of techniques.  The initial breakthrough was due to Bourgain, who introduced the method of `induction on energy' \cite{Bou1}.  This technique has been significantly developed and refined.  Presently, the typical approach to problems as in Theorem~\ref{T:scatter} follows the so-called `Kenig--Merle roadmap' developed in \cite{KenMer}.  One proceeds by contradiction:  Assuming the theorem to be false, one constructs a minimal energy counterexample, which (due to minimality) enjoys certain compactness properties.  One then shows that such compactness properties are at odds with the dispersive/conservative nature of the equation and ultimately lead to a contradiction; this is often achieved through the use of conservation laws together with certain nonlinear estimates known as virial or Morawetz estimates.  For an expository introduction to these techniques, we refer the reader to \cite{KV-Clay, V-Oberwolfach}. 

Beginning with the work of Kenig and Merle \cite{KenMer2}, a great deal of recent research has focused on establishing analogous results beyond the mass- and energy-critical cases.  In such cases, the `Kenig--Merle roadmap' naturally leads to a proof of scattering under the assumption of \emph{a priori} bounds in the critical Sobolev space, where the assumed bounds play the role of the `missing conservation law' at critical regularity.  Stated roughly, we have the following conjecture:
\begin{conjecture}
For the defocusing NLS or NLW, any solution that remains bounded in the critical Sobolev space is global-in-time and scatters.
\end{conjecture}
 By now, the range of positive results of this type is extensive. For the case of NLS, see \cite{KenMer2, KV4, KMMV, KMMV2, Mur1, Mur2, Mur3, XieFan, DMMZ, MMZ, LuZheng, Gao, Tengfei}; for the case of NLW, see \cite{BADM, Bulut1, Bulut2, Bulut3, DuyKenMer, DuyRoy, DuyYang, DodLaw, KenMer3, KV5, KV6, Rodriguez, Shen1, Shen2}.  While some recent remarkable work of Dodson \cite{Dodson-new1, Dodson-new2} has actually established \emph{unconditional} scattering results at critical regularity for the energy-subcritical NLW with radial initial data, the majority of the scattering results for large data at `non-conserved' critical regularity are conditional in nature.  We would also like to mention the recent work of D'Ancona \cite{D'Ancona}, who has established some well-posedness results for the defocusing energy-supercritical NLW in the exterior of a ball.  

In this paper, we carry out numerical simulations for the energy-supercritical NLW with radial (i.e. spherically symmetric) initial conditions\footnote{Note that radiality is preserved in time. This is a consequence of the fact that the Laplacian commutes with rotations, together with the uniqueness of solutions.}, where \emph{energy-supercritical} refers to the condition $p>\tfrac{4}{d-2}$, or equivalently $s_c>1$.  In the radial setting, \eqref{nlw} takes the form
\begin{equation}\label{NLW-ode}
-\partial_t^2 u + \partial_r^2 u + \tfrac{d-1}{r}\partial_r u = \mu |u|^p u,\quad u:\R_t\times(0,\infty)\to\R,
\end{equation}
where we write $u=u(t,r)$, with $r >0$ and impose the Neumann boundary condition $\partial_r u|_{r=0}\equiv 0$.  As in \cite{2659737, StraussVazquez}, the restriction to radial solutions provides a significant simplification in the numerical analysis of \eqref{nlw}.  Our main result is to demonstrate (numerically) boundedness of the critical Sobolev norms for large time, thus lending support to the conditional scattering results discussed above.  We expect that similar results will hold in the non-radial setting and plan to address this case in future work. 

For the sake of concreteness, we focus on two representative cases, namely,
\[
(d,p, s_c) = \big(3,6, \tfrac76\big)\qtq{and}(d,p, s_c) = \big(5,2, \tfrac32\big).
\]
These particular cases were considered in the works \cite{KenMer3, KV5, KV6, Bulut2, Bulut3}, which established scattering under the assumption of \emph{a priori} bounds for $(u,\partial_t u)$ in $\dot H^{s_c}\times \dot H^{s_c-1}$. We study a range of choices for $u_0 = u(0, r)$ and $u_1 = \partial_t u(0, r)$ (see Section~\ref{S:Results}), and in all cases we observe (numerically) that the critical Sobolev norm converges after a short time and, in particular, remains bounded for large times.  Additionally, we compute numerically the potential energy (i.e. the $L^{p+2}$ norm), the $L^\infty$ norm, and certain scale-invariant Besov norms.  We observe that the Besov norms become relatively small (compared to the Sobolev norms), and that the higher Lebesgue norms decay at a rate that matches solutions to the linear wave equation.  All of this behavior is consistent with scattering (see e.g. Section~\ref{S:Scattering}).  We describe our results in detail in Section~\ref{S:Results}.

We would also like to compare our results with the work of Strauss and Vazquez \cite{StraussVazquez} on the closely-related defocusing nonlinear Klein--Gordon equation (NLKG)
\begin{equation}\label{NLKG}
-\partial_t^2 u + \Delta u - u = |u|^{p} u,
\end{equation}
which they studied in dimension $d=3$ with radial initial conditions. They considered several choices of nonlinearity, including $p\in\{2,4,6,8\}$ as well as the nonlinearity $\sinh(5u)-5u$.  They also observed numerically that solutions should decay in $L^\infty$ at a rate matching solutions to the underlying linear problem.  Our results may be viewed in part as an extension of those in \cite{StraussVazquez} (albeit in the setting of NLW, rather than NLKG).  Indeed, we also observe the sharp $L^\infty$ decay in the NLW setting; additionally, we have simulated the solutions over a longer time interval and have computed several other quantities that are related to the problem of scattering (e.g. the Sobolev and Besov norms).  We have also treated the case $d=5$, in addition to the three-dimensional case.

At present, existing analytic techniques are generally insufficient to rigorously establish the boundedness in time of the critical Sobolev norm of solutions, unless such norms can be controlled by conserved quantities (although we should mention again the remarkable work of \cite{Dodson-new1, Dodson-new2} for the case of the radial energy-subcritical NLW).  In particular, this problem seems to be especially challenging in the energy-supercritical regime, as there is no known coercive conserved quantity above the regularity of the energy.  In general, such boundedness has generally been expected to hold true in the defocusing setting, as both the dispersion of the underlying linear equation and the defocusing nature of the nonlinearity tend to cause solutions to spread out and decay.  Our results provide additional numerical evidence in support of the belief of boundedness and lend support to the wide range of conditional scattering results for NLW that have been established in recent years.  

It remains an important open problem in the analysis of nonlinear dispersive PDE to determine whether energy-supercritical Sobolev norms do indeed remain bounded in time in general.  In the \emph{negative} direction, we would like to mention the recent preprint of Merle, Rapha\"el, Rodnianski, and Szeftel \cite{MRRS}, who have constructed radial solutions to the defocusing energy-supercritical NLS with energy-supercritical Sobolev norms blowing up in finite time at a polynomial rate!  Their result relies on the hydrodynamical formulation of NLS, making use of suitable underlying dynamics for the  compressible Euler equations in order to produce a highly oscillatory blowup profile. 

The rest of this paper is organized as follows: In Section~\ref{S:Notation}, we collect some basic notation and preliminaries.  In Section~\ref{S:Methods}, we describe the numerical methods we use in this work.  In Section~\ref{S:Results}, we describe  the sets of initial conditions used in the numerical simulations and discuss our numerical findings.  In Section~\ref{S:Scattering}, we prove a simple scattering result (namely, scattering holds if the critical Besov norm is sufficiently small compared to the critical Sobolev norm), which is relevant to the discussion in Section~\ref{S:Results}.  Finally, in Appendix~\ref{Incoming} we discuss the notion of incoming/outgoing waves, which play a role in our choice of initial conditions.

\subsection{Notation and preliminaries}\label{S:Notation}

We use the standard notation for Lebesgue norms, e.g.
\[
\|u\|_{L_x^\rho(\R^d)} = \biggl(\int_{\R^d} |u(x)|^\rho\,dx\biggr)^{\frac{1}{\rho}}
\]
for $1\leq \rho<\infty$.  We denote space-time norms by $L_t^q L_x^\rho$, i.e.
\[
\|u\|_{L_t^q L_x^\rho(I\times\R^d)} = \bigl\|\,\|u(t)\|_{L_x^\rho(\R^d)}\,\|_{L_t^q(I)}. 
\]

We define Sobolev and Besov norms by utilizing the Fourier transform, denoted by $\F f$ or $\hat f$.  We define
\[
\hat f(\xi)=(2\pi)^{-\frac{d}{2}}\int_{\R^d} e^{-ix\xi}f(x)\,dx,\qtq{so that} f(x) = (2\pi)^{-\frac{d}{2}}\int_{\R^d} e^{ix\xi}\hat f(\xi)\,d\xi. 
\]
The homogeneous $L^2$-based Sobolev spaces are then defined by
\[
\| u\|_{\dot H_x^s(\R^d)} = \| |\nabla|^s u\|_{L_x^2(\R^d)} = \| |\xi|^s \hat u\|_{L_\xi^2(\R^d)}. 
\]
Besov spaces are defined using the standard Littlewood--Paley multipliers.  In particular, for $N\in 2^{\mathbb{Z}}$ we let $\varphi_N$ denote a smooth bump function supported where $|\xi|\sim N$, with $\sum\varphi_N \equiv 1$.  We then define the Littlewood--Paley projections $P_N u$ through the Fourier transform, i.e. 
\[
P_N u = \F^{-1} \varphi_N \hat u.
\]
The Besov norm $\dot B^{s}_{q,\alpha}$ is defined via
\[
\|u\|_{\dot B^s_{q,\alpha}(\R^d)} = \bigl\| \|N^{s} P_N u\|_{L^q(\R^d)} \bigr\|_{\ell_N^\alpha(2^{\mathbb{Z}})}. 
\]

We will frequently consider the Sobolev norm
\[
\|(u,\partial_t u)\|_{\dot H_x^{s_c}\times \dot H_{x}^{s_c-1}}^2 := \| u\|_{\dot H_x^{s_c}}^2+\|\partial_t u \|_{\dot H_x^{s_c-1}}^2,\qtq{where} s_c>1,
\]
as well as the Besov norms
\[
\|u\|_{\dot B^{s_c}_{2,\infty}} = \sup_{N\in 2^{\mathbb{Z}}} N^{s_c}\|P_N u\|_{L_x^2} \qtq{and} \|\partial_t u \|_{\dot B_{2,\infty}^{s_c-1}} = \sup_{N\in 2^{\mathbb{Z}}} N^{s_c-1}\|P_N \partial_t u \|_{L_x^2}.
\]
Note that for any $s$, we have 
\[
N^{s}\| \varphi_N \hat u\|_{L_\xi^2} \lesssim \|u\|_{\dot H^{s}}
\]
uniformly in $N$, which implies 
\[
\|u\|_{\dot B_{2,\infty}^{s_c}} \lesssim \|u\|_{\dot H^{s_c}} \qtq{and}\|\partial_t u\|_{\dot B_{2,\infty}^{s_c-1}} \lesssim \|\partial_t u\|_{\dot H^{s_c-1}}.
\]

As mentioned above, the restriction to radial (i.e. spherically symmetric) solutions leads to some simplifications.  We have already mentioned the simplification of the PDE (and hence the numerical analysis).  Additionally, we may change to spherical coordinates and write
\begin{align*}
\hat u(\xi) & = (2\pi)^{-\frac{d}{2}}\int_0^\infty \biggl[\int_{\partial B(0,1)} e^{-i|\xi|r\frac{\xi}{|\xi|}\cdot \omega}\,dS(\omega)\biggr]u(r) r^{d-1}\,dr \\
& = |\xi|^{-\frac{d-2}{2}}\int_0^\infty J_{\frac{d-2}{2}}(r|\xi|)u(r) r^{\frac{d}{2}}\,dr,
\end{align*}
where $J_\nu$ denotes the standard Bessel function (see e.g. \cite{SteinWeiss}).  This informs our numerical computation of the Fourier transform, and therefore the relevant Sobolev and Besov norms.

\section{Numerical methods}\label{S:Methods}

We first truncate (\ref{NLW-ode}) to a computational domain $[0, R_{\rm max}]$ with $R_{\rm max}$ sufficiently large that the truncating effect can be neglected. 
We then reformulate (\ref{NLW-ode}) and solve the following problem: 
\begin{eqnarray}\label{eq2}
\partial_{tt} u = \frac{1}{r^{d-1}}\frac{\partial}{\partial r}\Big(r^{d-1}\frac{\partial u}{\partial r}\Big) - \mu|u|^p u, && \mbox{for} \ r \in (0, R_{\max}), \ \ \ t > 0,\qquad \\
\label{BC}
\partial_r u(t, r)|_{r = 0} = 0, \quad\ \  u(t,R_{\max}) = 0, && \mbox{for} \ t\ge 0, \\
\label{IC}
u(0, r) = u_0(r), \quad \ \ \partial_t u(t, r)\mid_{t = 0} = u_1(r) && \mbox{for} \ r \in [0, R_{\max}]. 
\end{eqnarray}
Here the homogeneous  Dirichlet boundary conditions are considered at $r = R_{\max}$. 

In \cite{StraussVazquez},  an implicit  finite difference method is introduced to solve the NLKG (\ref{NLKG}).  This method exactly conserves the discrete energy, but at each time step one has to solve a nonlinear system with iteration methods (e.g. Newton's method).  Hence, its computational costs are high.  To avoid solving nonlinear systems at each time step, explicit methods are popular in practice.  In \cite{DonSch}, an explicit method with a second-order difference scheme for both temporal and spatial discretization is used to study the boundary of the forward scattering region for NLKG.  In \cite{2659737}, the authors study the question of boundedness of critical Sobolev norms for NLS using a finite difference scheme in space and the explicit fourth-order Runge--Kutta method in time.  Compared to implicit methods, these explicit methods are easier to implement and have less computational costs.  They do not have the exact energy conservation, but could provide a good approximation to it if proper numerical parameters are used in simulations \cite{DonSch, 2659737}. 

Next, we introduce a second-order finite difference method to solve (\ref{eq2})--(\ref{IC}).  Our method is different from that in \cite{DonSch} where the change of variable $v = ru$ is introduced to move the singularity term from linear to nonlinear part.  Instead, we will directly discretize the equation (\ref{eq2}).  Denote the mesh size $\Delta r = R_{\max}/N$ with $N$ a positive integer, and let $\Delta t > 0$ be the time step.  Then we define the spatial grid points and time sequence as 
\begin{equation}
r_j = j\Delta r, \quad \mbox{for} \ \  j = 0, 1, \ldots, N; \qquad t_n = n\Delta t, \quad \mbox{for} \ \ n = 0, 1, \ldots.\nonumber
\end{equation}
We denote the numerical approximation of $u(t_n,r_j)$  by $U_j^n$.  Then the equation (\ref{eq2}) can be approximated by the difference scheme: 
\begin{equation} \label{eq-dis}
\begin{aligned}
&\frac{U_j^{n+1} - 2U_j^n+ U_j^{n-1}}{\Delta t^2} \\
&\quad\quad=\frac{(\eta_j^+)^{d-1}\big(U_{j+1}^n - U_j^n\big) + (\eta_j^-)^{d-1}\big(U_{j-1}^n - U_j^n\big)}{\Delta r^2} - \mu |U_j^n|^pU_j^n,
\end{aligned}
\end{equation}
for $1 \le j \le N-1$ and $n \ge 1$, where we denote $\eta_j^\pm = (r_j \pm {h}/{2}) /r_j$. 
For $j = 0$, the approximation is given by: 
\begin{eqnarray}
\frac{U_0^{n+1} - 2U_0^n + U_0^{n-1}}{\Delta t^2} = \frac{d \big(U_1^n-2U_0^n + U_{-1}^n\big)}{\Delta r^2}- \mu |U_0^n|^pU_0^n, \qquad n \ge 1,
\end{eqnarray}
where $U_{-1}$ denotes the solution at the ghost point $r = r_{-1}$. 

The discretization of the boundary conditions (\ref{BC}) gives
\begin{eqnarray}
\frac{U_1^n - U_{-1}^n}{2\Delta r} = 0, \quad\ \ U_N^n = 0, \qquad n \ge 0,
\end{eqnarray}
which implies that $U_{-1}^n \equiv U_1^n$ for $n \ge 0$.  At $t = 0$, the initial condition (\ref{IC}) is discretized as 
\begin{equation}\label{DIC}
U_j^0 = u_0(r_j), \quad \ \ \frac{U_j^1 - U_j^0}{\Delta t} = u_1(r_j), \qquad \mbox{for} \ \ 0 \le j \le N-1.
\end{equation}

Combining (\ref{eq-dis})--(\ref{DIC}) yields an explicit second-order finite difference scheme to the problem (\ref{eq2})--(\ref{BC}).  
In the simulations, one sufficient stability condition  of this scheme is
\begin{equation}
\Delta t \le  \Delta r\,\sqrt{\frac{2^{d-1}}{1+3^{d-1}}}.
\end{equation}
This suggests that the simulations of (\ref{NLW-ode}) for higher dimensions are generally more time-consuming, as a smaller time step is required to ensure the numerical stability. 

The discrete energy is calculated by
\begin{equation}
E^n = \frac{\Delta r}{2}\sum_{j=1}^N \bigg[\Big(\frac{U_j^{n+1}-U_j^{n-1}}{2\Delta t}\Big)^2 
+ \Big(\frac{U_{j+1}^n - U_{j-1}^n}{2\Delta r}\Big)^2 + \frac{2\mu}{p+2}|U_j^n|^{p+2}\bigg]r_j^{d-1}
\end{equation}
for $n \ge 1$, while the Besov norms are calculated following the discretization method in \cite{2659737}.

\subsection{Verification of numerical method}
\label{section2-1}

We will test the performance of our  method for different numerical parameters (i.e. $R_{\rm max}$, $\Delta r$, and $\Delta t$).

First, we study the truncating effect of $R_{\max}$ by fixing the time step $\Delta t$ and mesh size $\Delta r$.   
In general, a large computational domain $[0,  R_{\max}]$ is required to avoid the artifacts from truncation. 
However, the larger the computational domain, the larger the number of unknowns, and thus the larger the computational costs. 
\begin{table}[htb!]
\begin{center}
\begin{tabular}{|c|c|ccc|}
\hline 
Time & $R_{\max}$ & $u(t, r=0)$  & $u(t, r=19)$ & $\displaystyle \max_{r\in [0, R_{\max}]}|u(t, r)|$\\ 
\hline
\multirow{3}{*}{$t = 5$} & $R_{\max} = 20$ & 0.00417100 & 0.00000000 &0.30992600\\ 
&$R_{\max} = 30$ &0.00417100 & 0.00000000 &0.30992600\\ 
& $R_{\max} = 50$ & 0.00417100 & 0.00000000 &0.30992600\\ 
\hline 
\multirow{3}{*}{$t = 10$} & $R_{\max} = 20$ & 0.00008800 & 0.00000000 & 0.15864800\\ 
&$R_{\max} = 30$ & 0.00008800 & 0.00000000 & 0.15864800\\
& $R_{\max} = 50$ & 0.00008800 & 0.00000000 & 0.15864800\\
\hline 
\multirow{3}{*}{$t = 15$} & $R_{\max} = 20$ & 0.00000900 & 0.00000000 & 0.10659100\\
&$R_{\max} = 30$ & 0.00000900 & 0.00000000 & 0.10659100\\
& $R_{\max} = 50$ & 0.00000900 & 0.00000000 & 0.10659100\\ 
\hline
\end{tabular}
\caption{Comparison of numerical solutions for various $R_{\max}$ where $d = 3$, $p = 6$, and $u_0 = 4\exp(-r^2)$ and $u_1 = 0$ in (\ref{IC}).}\label{Table1}
\end{center}
\end{table}
In Table \ref{Table1}, we compare the numerical solutions $u$ that are computed with different choices of  $R_{\max}$.  For the different $R_{\max}$ chosen here, our method yields the same solutions of $u(t, r=0)$ and $\max_{r\in[0, R_{\max}]}|u(t, r)|$. 
Moreover,  the solution $u(t, r=19)$ remains zero even through time $t = 15$, suggesting that $R_{\max} = 20$ is large enough for the computation time $t\in[0, 15]$.

Next, we will fix $R_{\rm max} = 20$  and study the numerical errors for different $\Delta t$ and $\Delta r$. 
Since the exact solution of (\ref{NLW-ode}) is unknown, we will use the numerical solution with a very fine mesh (here $\Delta r = 2^{-12} $ and $\Delta t = 2^{-14}$) as a proxy for the ``exact" solution when computing the numerical errors. 
\begin{figure}[!htbp]
\centering
\begin{subfigure}{.5\textwidth}
\centering
a)\includegraphics[width=1\linewidth]{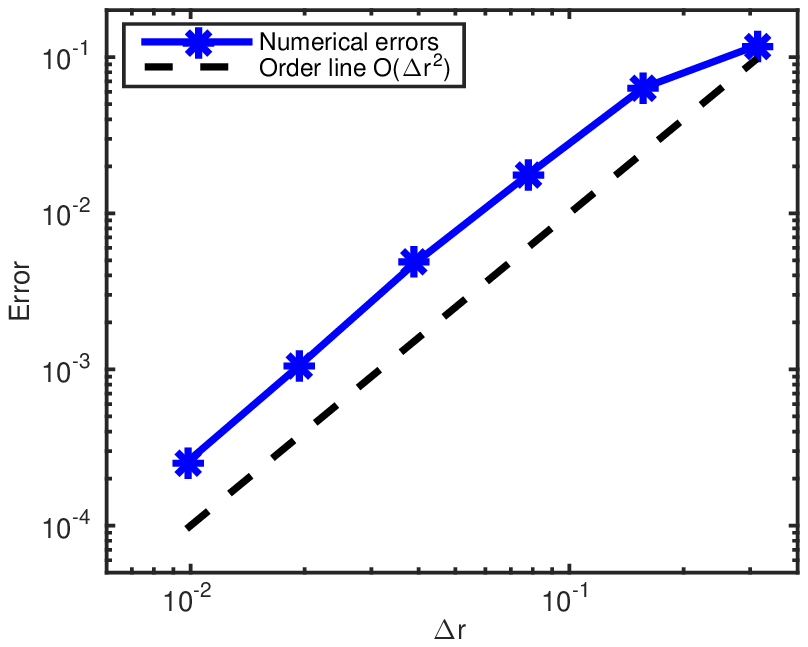}
\end{subfigure}%
\begin{subfigure}{.5\textwidth}
\centering
b)\includegraphics[width=1\linewidth]{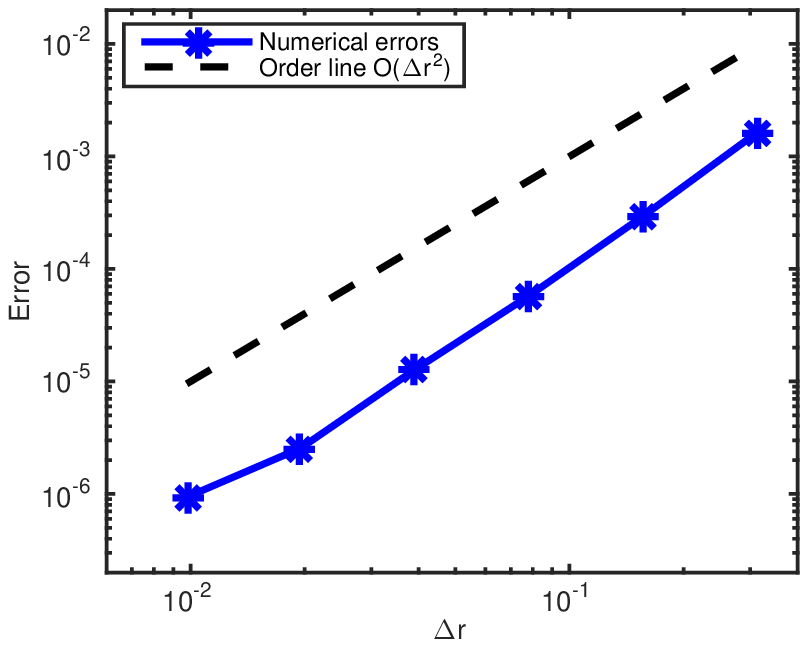}
\end{subfigure}
\caption{Numerical errors for different $\Delta t$ and $\Delta r$, where we choose $\Delta t = \Delta r/4$. The initial condition is $u_0 = 4\exp(-r^2)$ and $u_1 = 0$ in (\ref{IC}). a) $d = 3$ and $p = 6$; b) $d = 5$ and $p = 2$.}\label{EFig1}
\end{figure}
Figure \ref{EFig1} shows the $\ell^2$-norm errors in solution $u(2, r)$, where an order line of $O(\Delta r^2)$ is included for easy comparison. 
We choose the time step and mesh size to satisfy $\Delta t = \Delta r/4$. 
The results in Figure \ref{EFig1} confirm the second-order accuracy of our methods. 
It additionally shows that for fixed numerical parameters, the numerical errors for $d = 5$ are smaller than those of $d = 3$. For example, the numerical error of $\Delta r = 0.0098$ is $2.54\times10^{-4}$ for $d = 3$ and $9.38\times 10^{-7}$ for $d = 5$.

\section{Numerical results}\label{S:Results}

We now discuss our numerical results.  To begin, let us discuss the specific choices of initial conditions used in this paper.  The key assumption on the data is that of spherical symmetry, which reduces the equation to a one-dimensional problem and thereby greatly simplifies the numerical analysis.  
We begin by considering the cases of Gaussians that are large enough to be safely outside of the small data regime (for otherwise scattering is a known consequence of the small-data well-posedness theory).  As in \cite{2659737}, we would also like to consider some initial data for which the underlying linear equation would experience some initial focusing toward the origin; for such data, we would then like to observe that the defocusing nonlinearity counters this effect.  The authors of \cite{2659737} achieved this by multiplying the Gaussian initial data by $e^{\alpha ir^2}$ for some $\alpha>0$.  In the setting of the radial wave equation, it seems natural to consider `incoming' initial data for this purpose (see e.g. \cite{Marius}), which in our setting refers to the condition
\begin{equation}\label{incoming-condition}
u_1 = \partial_r u_0 + \tfrac{d-2}{r}u_0,
\end{equation} 
where $u_0=u|_{t=0}$ and $u_1=\partial_t u|_{t=0}.$ We discuss the origin of this condition in more detail in Appendix~\ref{Incoming}.

Apart from choosing incoming/outgoing initial data (which requires defining $u_1$ precisely in terms of $u_0$), we found that varying the choice of initial velocity $u_1$ plays almost no role in terms of the long-time behavior of the solution.  Thus, other than the cases for which we take `incoming' data, we will be content to work with the simplest choice $u_1=0$.  Similar to \cite{2659737}, we will then choose Gaussian data ($u_0=C\exp(-r^2)$), `ring' data ($u_0=Cr^2\exp(-r^2)$), or an oscillating Gaussian ($u_0= C\exp(-r^2)\sin(ar)$). 

As far as the incoming condition, we note that \eqref{incoming-condition} includes the singular term $1/r$ in the expression for $u_1$.  While one can verify that $u_1$ still belongs to $L^2\cap\dot H^{s_c-1}$ (see Lemma~\ref{Lemma}), we found that unless $u_0$ vanishes to high enough order at $r=0$, it is difficult to simulate this condition numerically.  Thus we were led only to consider the incoming condition for the `ring' initial data $u_0=Cr^2\exp(-r^2)$ and for the oscillatory data $ u_0 = C\exp(-r^2)\sin(ar)$, for which cases the presence of $1/r$ is harmless.

Altogether we considered five choices of initial conditions, which are detailed in the subsections below.  As described in the introduction, for each case we consider the combinations $(d,p)=(3,6)$ and $(d,p)=(5,2).$  In the following, we summarize the quantities studied, as well as the corresponding findings. 
For each case, we study the time evolution of:
\begin{itemize}
\item[(i)] The solution $u(t, r)$.  We find that the solution decays over time and travels outward at a constant speed.

\item[(ii)] The critical Sobolev norms $\|u(t)\|_{\dot H^{s_c}}$ and  $\|\partial_t u(t)\|_{\dot H^{s_c-1}}$, with $s_c$ as in (\ref{sc}).  We find that the critical Sobolev norms may initially oscillate, but quickly settle down and converge to a constant. Moreover, our numerical results show that 
\[
\lim_{t\to\infty}\|u(t)\|_{\dot{H}^{s_c}} = \lim_{t\to\infty} \|\partial_t u(t)\|_{\dot{H}^{s_c-1}}.
\] 

\item[(iii)] The potential energy $\|u(t)\|_{L^{p+2}}$ and the supremum norm $\|u(t)\|_{L^\infty}$.  We find that both quantities tend to zero as $t\to\infty$.  More precisely, we observe
\[
\|u(t)\|_{L^{p+2}} \sim (1+t)^{-\frac{(d-1)p}{2(p+2)}} \qtq{and}\|u(t)\|_{L^\infty}\sim (1+t)^{-\frac{d-1}{2}},  
\]
for sufficiently large $t$, matching the decay rates for the underlying linear wave equation.

\item[(iv)] The critical Besov norms $\|u(t)\|_{\dot B_{2,\infty}^{s_c}}$ and $\|\partial_t u(t)\|_{\dot B_{2,\infty}^{s_c-1}}$.
We find that the Besov norms remain bounded and, in fact, become relatively small compared to the critical Sobolev norm as $t\to\infty$.  As discussed in Section~\ref{S:Scattering}, if solutions are sufficiently dispersed in frequency relative to their Sobolev norm, then one can prove scattering by a small-data type argument.
\end{itemize} 
{In what follows, we present our numerical results for our five representative cases, where we will always choose $R_{\max} = 20$, $\Delta r = 4e$-4, and $\Delta t = 1.25e$-4. As discussed previously, our explicit finite difference method does not exactly conserve the energy. However, Table \ref{Table2} shows that it has a good approximate conservation of energy for a long time (e.g. through time $t = 15$).
\begin{table}[htb!]
\begin{center}
\begin{tabular}{|c|cc|cc|}
\hline 
Case & \multicolumn{2}{c|}{$d = 3, \ p = 6$} & \multicolumn{2}{c|}{$d = 5, \ p = 2$}\\
\cline{2-5}
&$E(0)$ & $\Delta_{\max} E(t)/E(0)$ & $E(0)$ & $\Delta_{\max} E(t)/E(0)$  \\
\hline 
1 &  164.184 & 1.6468e-5 & 6.02927 & 2.3386e-7 \\
2 & 1980.14&  2.0506e-5 &  86.1537 & 1.7213e-7 \\
3 & 1996.30 &  1.9631e-5 & 128.380 & 1.7939e-7 \\
4 & 424.502 & 3.0777e-5 &11.8951 & 1.9420e-7\\
5 & 436.317 & 2.7890e-5  &22.3556 & 2.9567e-7 \\
\hline
\end{tabular}
\caption{Numerical illustration of approximate energy conservation, where $\Delta_{\max} E(t) := \max_{t\in[0, 15]}|E(t) - E(0)|$.}\label{Table2}
\end{center}
\end{table}

\subsection{Case 1. Gaussian data.}
In this case, we take {the initial condition} 
\begin{equation}\label{Gaussian}
u_0=4\exp(-r^2),\qquad u_1=0, \qquad\mbox{for \ \ $r \ge 0$.}
\end{equation}
{Figure \ref{Figure1u} presents the time evolution of the solution, which shows that the solution decays over time and} behaves essentially as a wave packet traveling {outward} with constant speed.

\begin{figure}[!htbp]
\centering
\begin{subfigure}{.5\textwidth}
\centering
\includegraphics[width=1\linewidth]{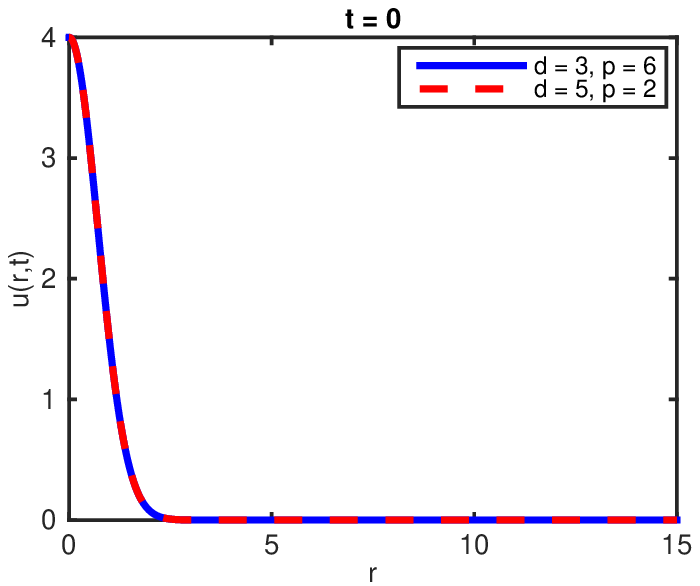}
\end{subfigure}%
\begin{subfigure}{.5\textwidth}
\centering
\includegraphics[width=1\linewidth]{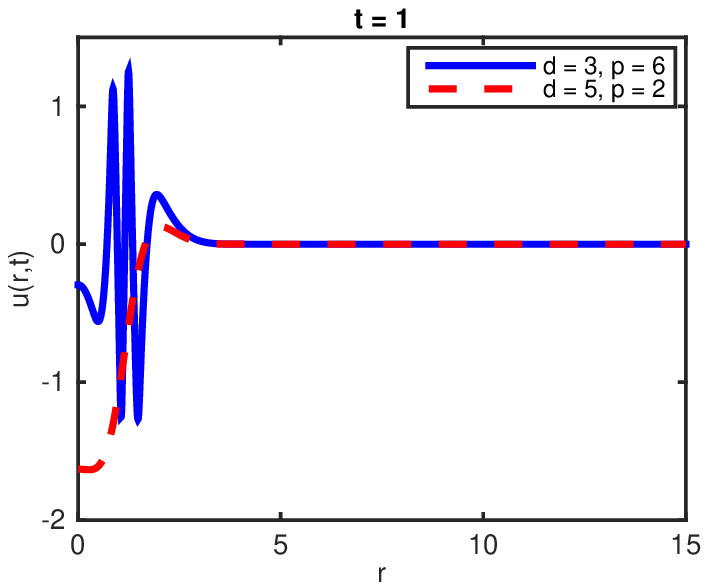}
\end{subfigure}
\begin{subfigure}{.5\textwidth}
\centering
\includegraphics[width=1\linewidth]{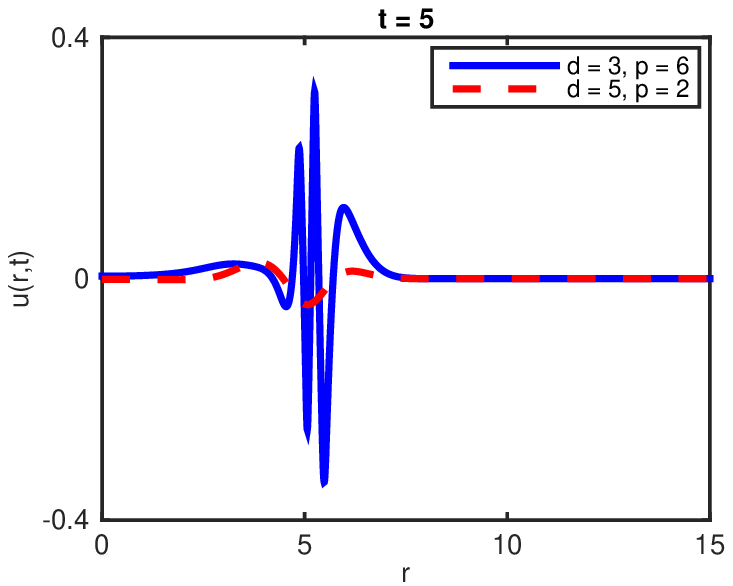}
\end{subfigure}%
\begin{subfigure}{.5\textwidth}
\centering
\includegraphics[width=1\linewidth]{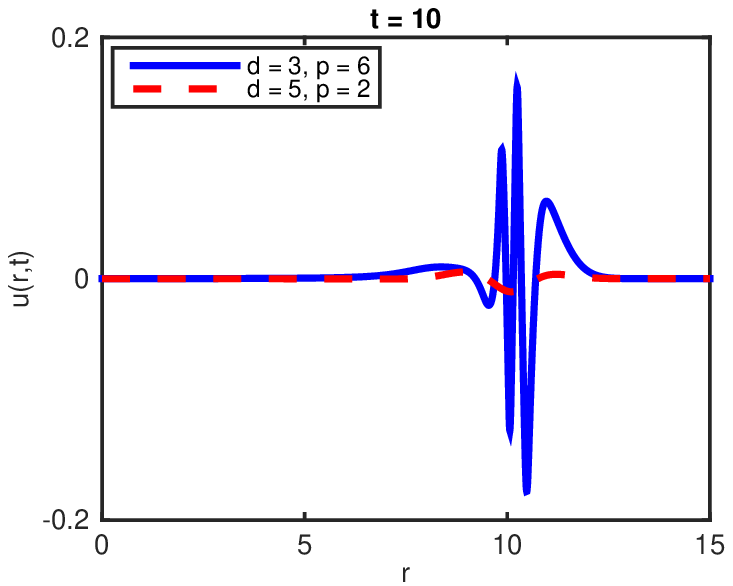}
\end{subfigure}
\caption{\footnotesize   Time evolution of the solution of NLW (\ref{NLW-ode}) with initial condition (\ref{Gaussian}).}\label{Figure1u}
\end{figure}

{Next, we further study} the decay of the solution $u$.  We would like to show decay of both the potential energy (i.e. the $L^{p+2}$-norm), as well as pointwise decay (i.e. the $L^\infty$-norm), both of which are  consistent with scattering.  In fact, we can show that the decay rate for these quantities matches the decay rate for solutions to the linear equation.  To see this, we {present in Figure \ref{Figure1d} the time evolution of the} following quantities and observe boundedness (in fact, convergence) for large $t$:
\begin{equation}\label{decay-plot}
(1+t)^{\frac{(d-1)p}{2(p+2)}}\|u(t)\|_{L^{p+2}}\qtq{and} t^{\frac{d-1}{2}}\|u(t)\|_{L^\infty}.
\end{equation}

\begin{figure}[!htbp]
\centering
\begin{subfigure}{.5\textwidth}
\centering
\includegraphics[width=1\linewidth]{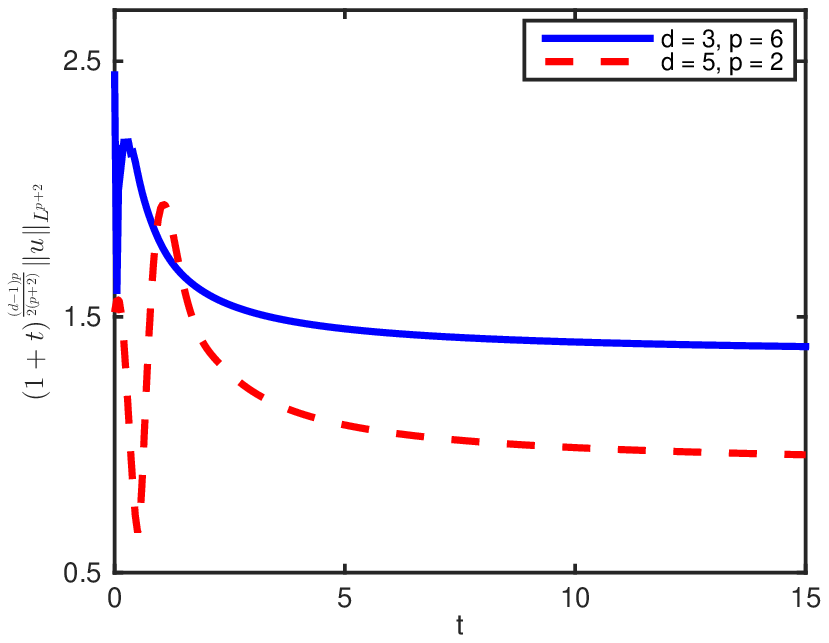}
\end{subfigure}%
\begin{subfigure}{.5\textwidth}
\centering
\includegraphics[width=1\linewidth]{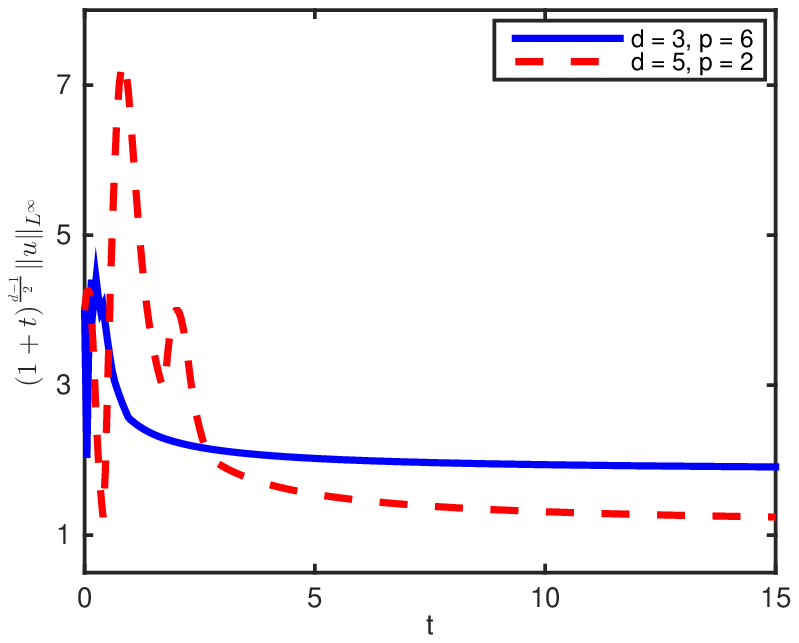}
\end{subfigure}
\caption{\footnotesize   Decay of higher norms in NLW with initial condition (\ref{Gaussian}).}\label{Figure1d}
\end{figure}

Figure \ref{Figure1hs} represents the main result of this paper, namely, numerical evidence for the boundedness of the critical Sobolev norms.  In fact, after some initial {oscillation}, we see that both the $\dot H^{s_c}$-norm of $u$ and the $\dot H^{s_c-1}$ norm of $\partial_t u$
 quickly settle down and converge.  {Moreover, our numerical results show that
\[
\lim_{t\to\infty}\|u(t)\|_{\dot{H}^{s_c}} = \lim_{t\to\infty}\|\partial_t u(t)\|_{\dot{H}^{s_c-1}}.
\]
}

\begin{figure}[!htbp]
\centering
\begin{subfigure}{.5\textwidth}
\centering
\includegraphics[width=1\linewidth]{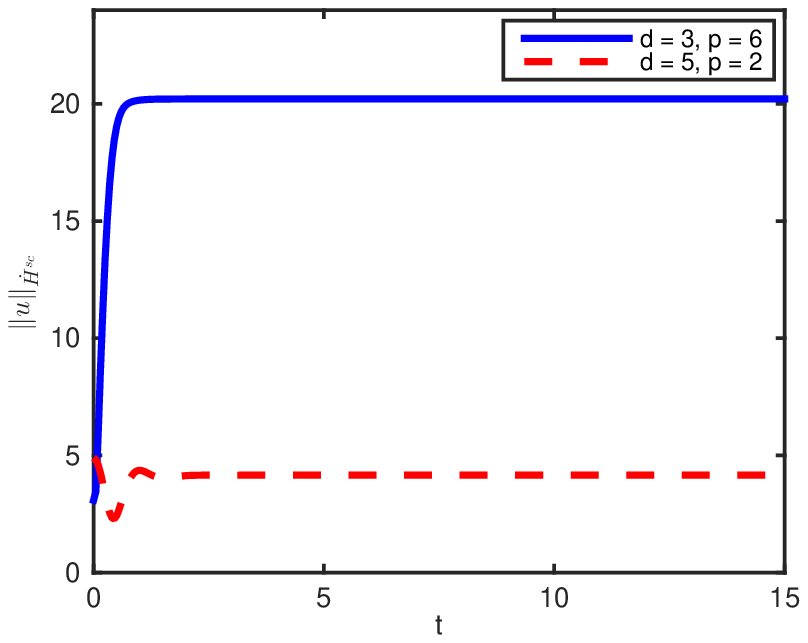}
\end{subfigure}%
\begin{subfigure}{.5\textwidth}
\centering
\includegraphics[width=1\linewidth]{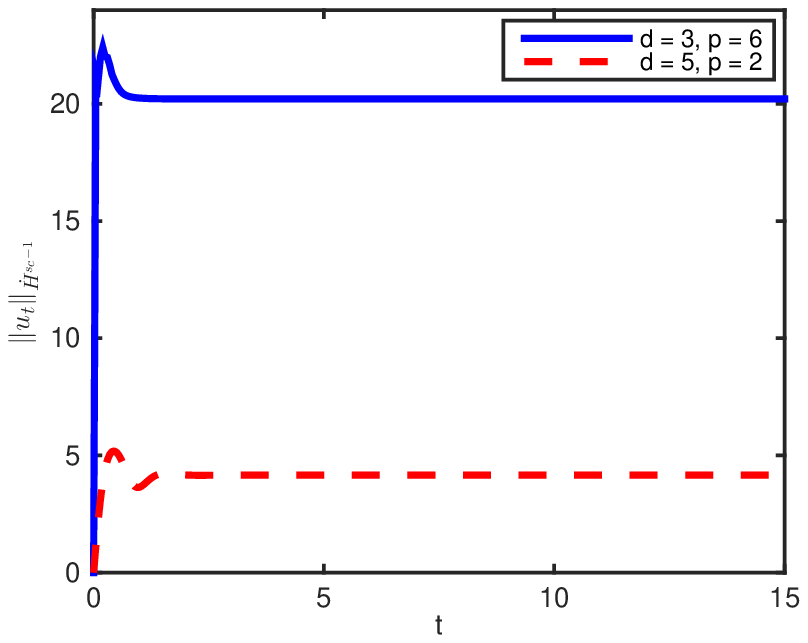}
\end{subfigure}
\caption{\footnotesize   Boundedness of critical Sobolev norms for NLW with initial condition (\ref{Gaussian}).}
\label{Figure1hs}
\end{figure}

Finally, we plot the critical Besov norms (namely, $\dot B^{s_c}_{2,\infty}$ for $u$ and $\dot B^{s_c-1}_{2,\infty}$ for $\partial_t u$) {in Figure \ref{Figure1b}}.  While these norms do not converge to zero, we can observe that they become relatively small compared to the critical Sobolev norms {in Figure \ref{Figure1hs}}.  As discussed in Section~\ref{S:Scattering}, it is possible to prove scattering in such a scenario.

\newpage

\begin{figure}[!htbp]
\centering
\begin{subfigure}{.5\textwidth}
\centering
\includegraphics[width=1\linewidth]{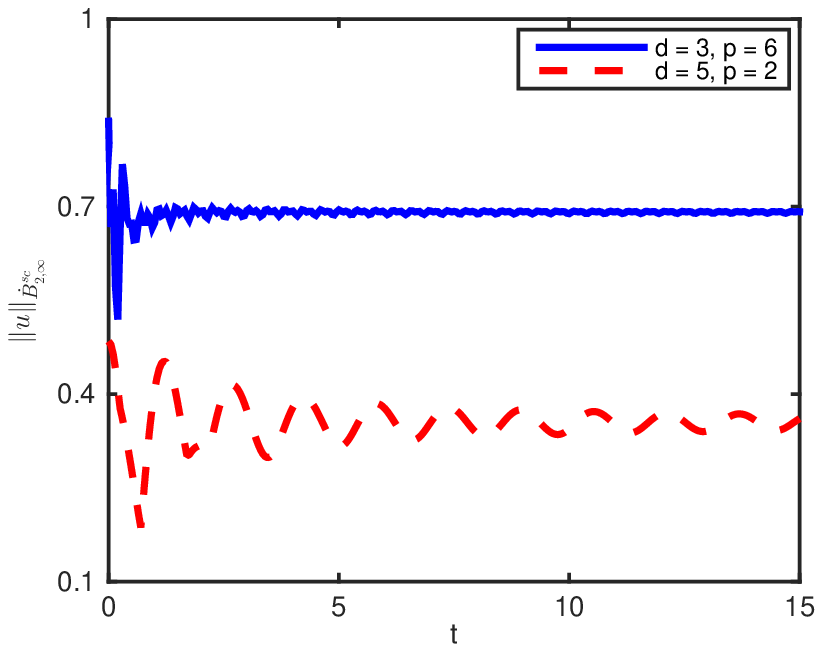}
\end{subfigure}%
\begin{subfigure}{.5\textwidth}
\centering
\includegraphics[width=1\linewidth]{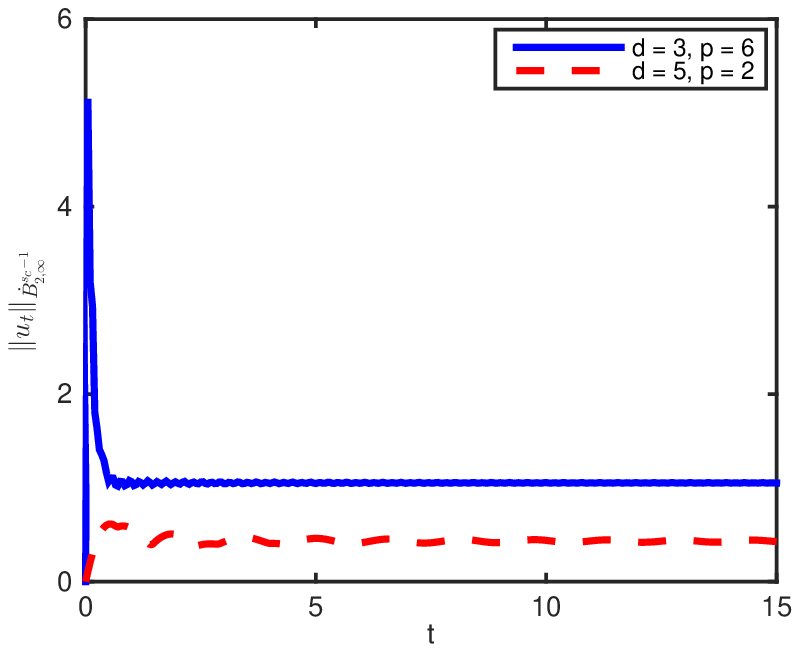}
\end{subfigure}
\caption{\footnotesize   Relative smallness Besov norms for NLW with initial condition (\ref{Gaussian}).}\label{Figure1b}
\end{figure}


\subsection{Case 2. Ring data.} In Case 2 we take
\begin{equation}\label{ring0}
u_0 = 10r^2\exp(-r^2),\qquad u_1=0, \qquad\mbox{for \ \ $r \ge 0$.}
\end{equation}
As we observed the same phenomena as in Case 1, we will be somewhat brief in our presentation.

{Figure \ref{Figure2u} presents the time evolution of the solution. It shows that the solution is more oscillating than the case with Gaussian initial conditions; however, similar to Case 1, the solution still decays over time and travels outward.}

\begin{figure}[!htbp]
\centering
\begin{subfigure}{.5\textwidth}
\centering
\includegraphics[width=1\linewidth]{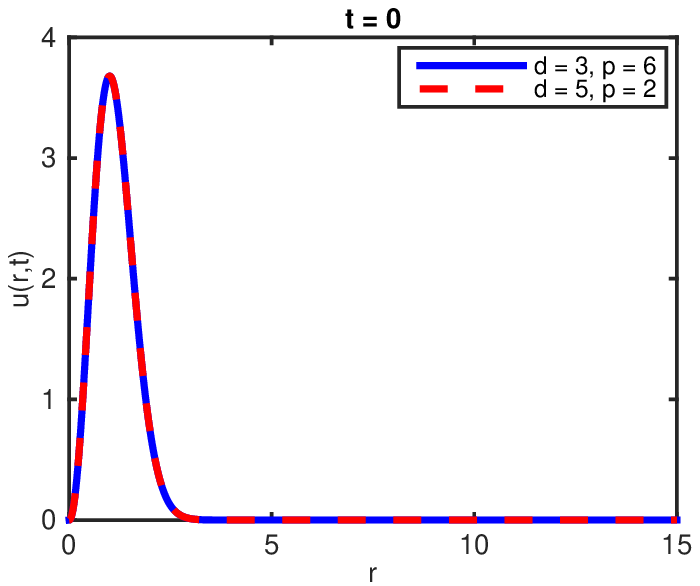}
\end{subfigure}%
\begin{subfigure}{.5\textwidth}
\centering
\includegraphics[width=1\linewidth]{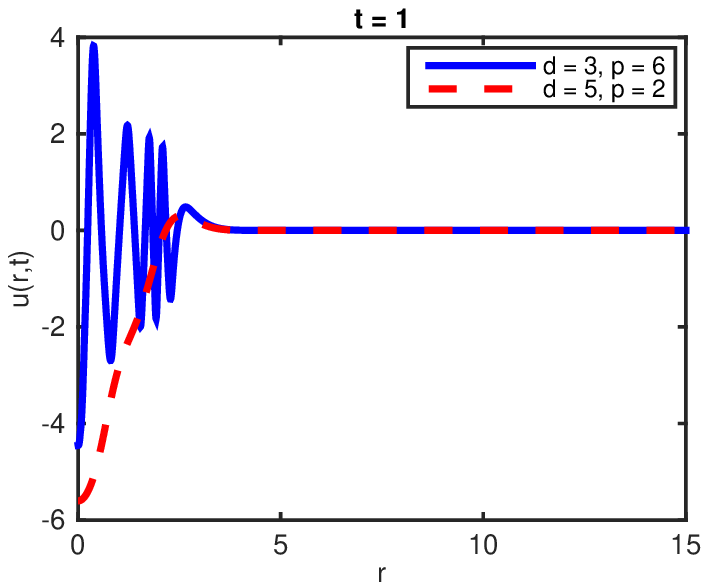}
\end{subfigure}
\begin{subfigure}{.5\textwidth}
\centering
\includegraphics[width=1\linewidth]{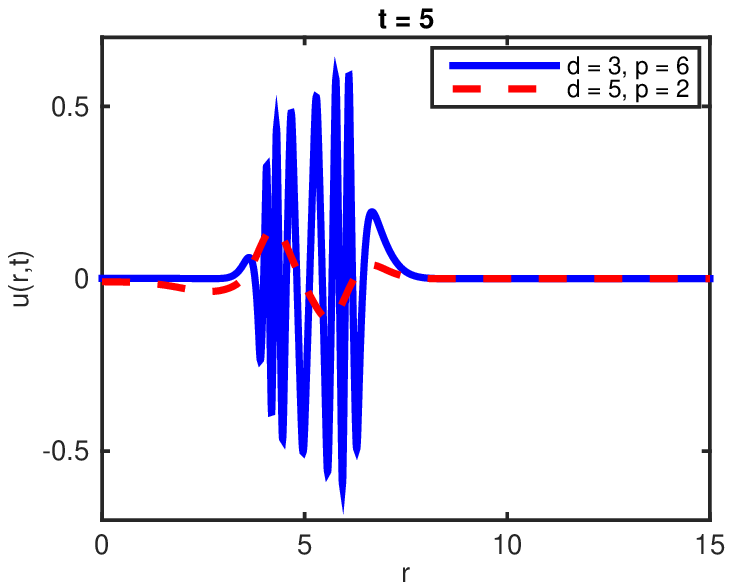}
\end{subfigure}%
\begin{subfigure}{.5\textwidth}
\centering
\includegraphics[width=1\linewidth]{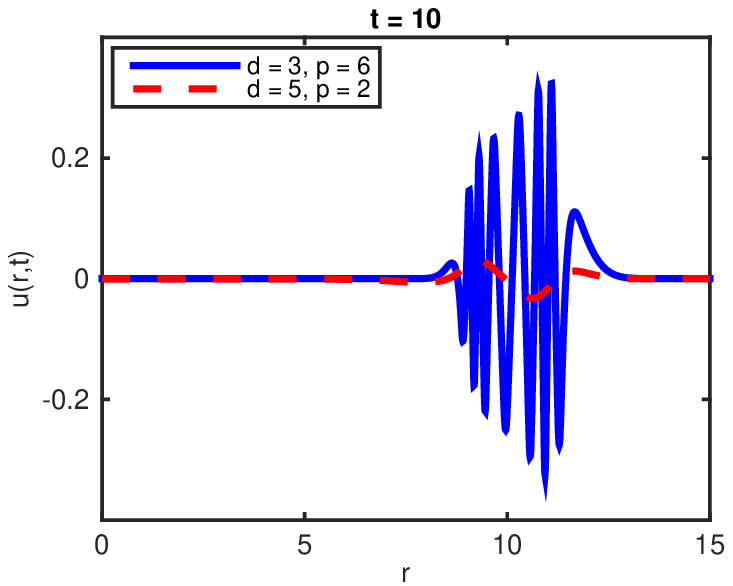}
\end{subfigure}
\caption{\footnotesize   Time evolution of solution of NLW (\ref{NLW-ode}) with initial condition (\ref{ring0}).}\label{Figure2u}
\end{figure}

{Figure \ref{Figure2norm} presents the critical Sobolev norms, Besov norms, and higher Lebesgue norms over time. It shows that the critical Sobolev norms $\|u\|_{\dot{H}^{s_c}}$ and $\|\partial_t u\|_{\dot{H}^{s_c-1}}$ quickly converge and stay at the same constant after a relatively small time, while the Besov norms are eventually bounded by a relatively small constant compared to the Sobolev norms.  The time evolution of the higher Lebesgue norms shows that the decay of $\|u\|_{L^{p+2}}$ and $\|u\|_{L^\infty}$ are qualitatively the same as in Case 1.
}

\begin{figure}[!htbp]
\centering
\begin{subfigure}{.5\textwidth}
\centering
\includegraphics[width=1\linewidth]{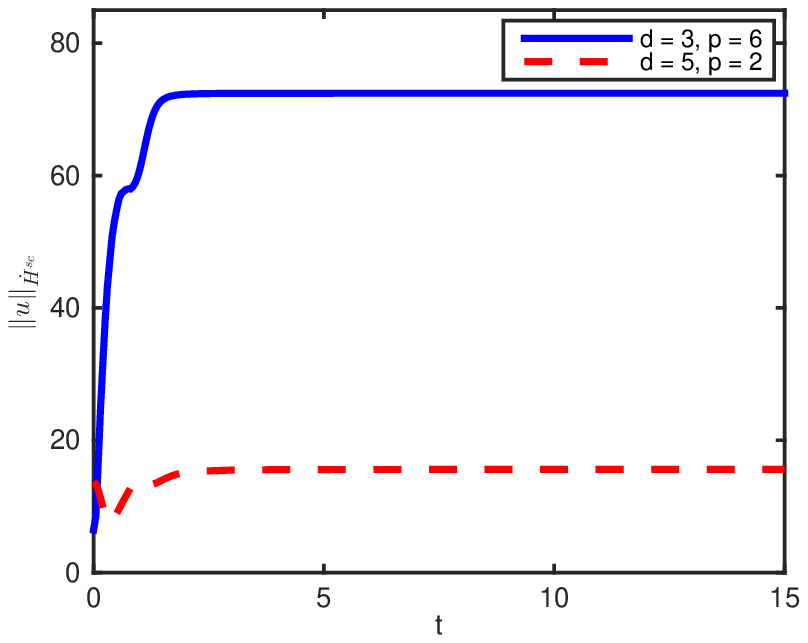}
\end{subfigure}%
\begin{subfigure}{.5\textwidth}
\centering
\includegraphics[width=1\linewidth]{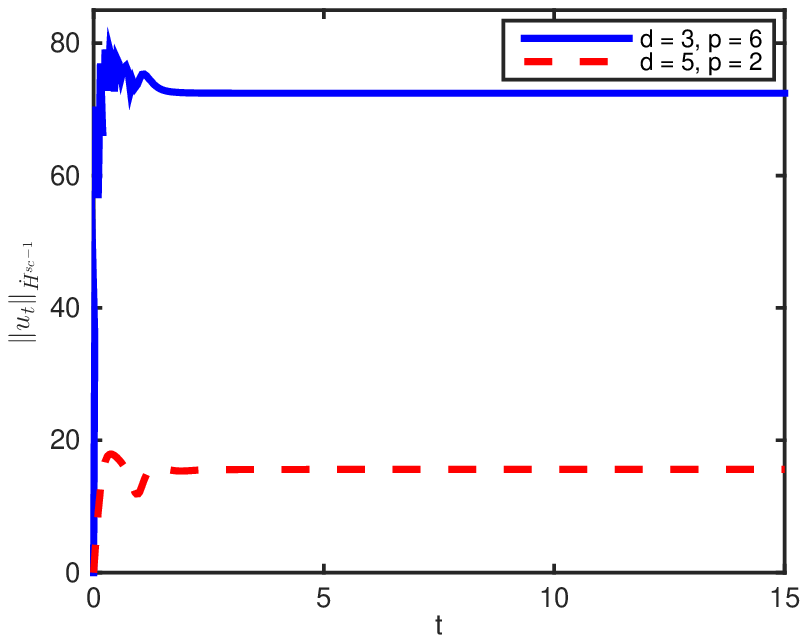}
\end{subfigure}
\begin{subfigure}{.5\textwidth}
\centering
\includegraphics[width=1\linewidth]{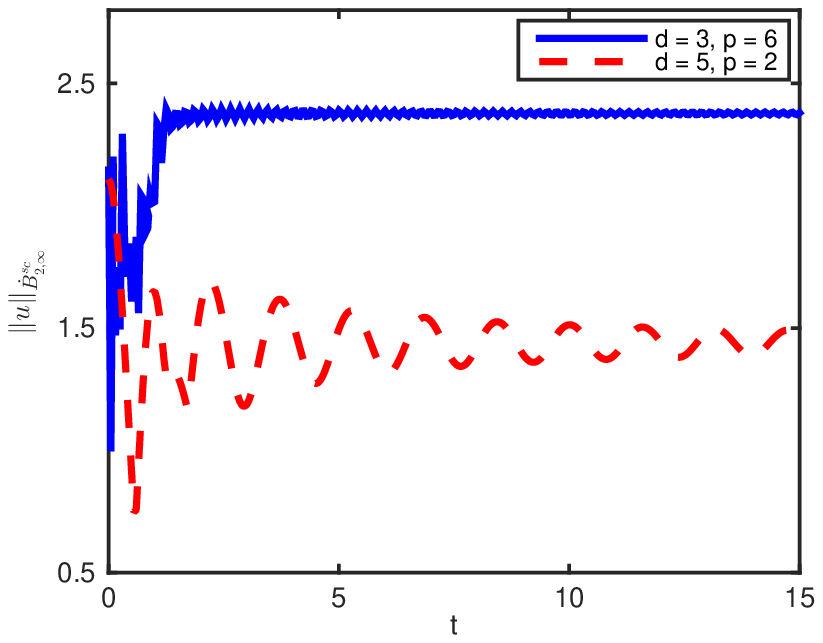}
\end{subfigure}%
\begin{subfigure}{.5\textwidth}
\centering
\includegraphics[width=1\linewidth]{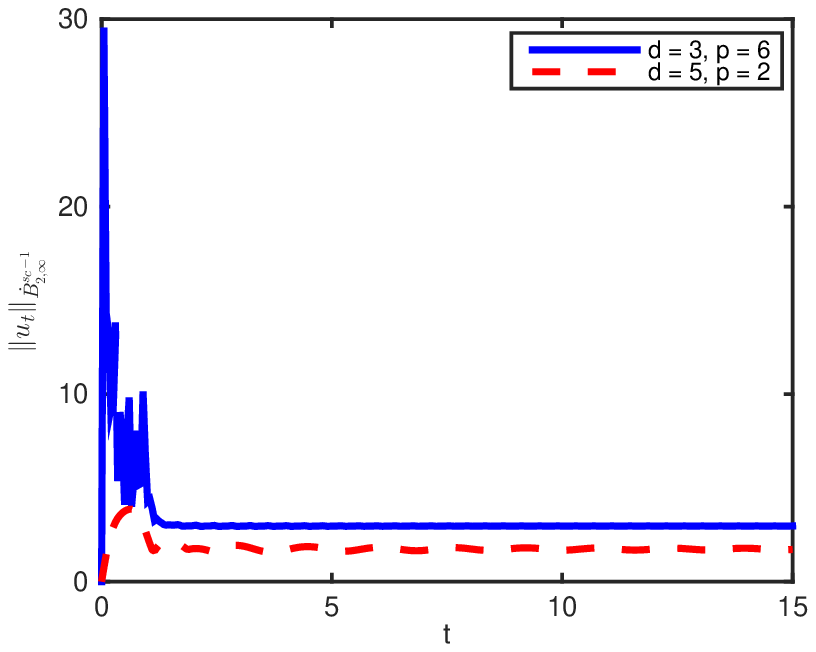}
\end{subfigure}
\begin{subfigure}{.5\textwidth}
\centering
\includegraphics[width=1\linewidth]{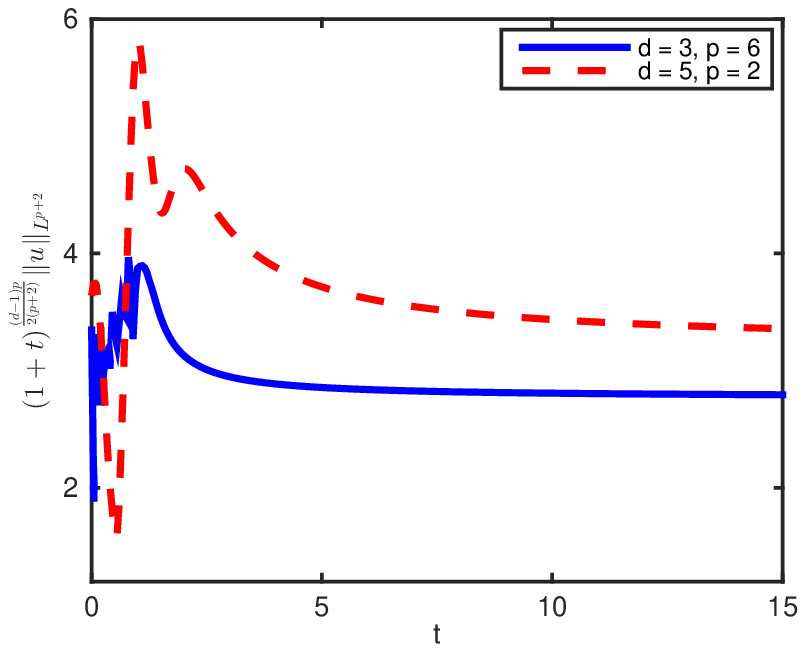}
\end{subfigure}%
\begin{subfigure}{.5\textwidth}
\centering
\includegraphics[width=1\linewidth]{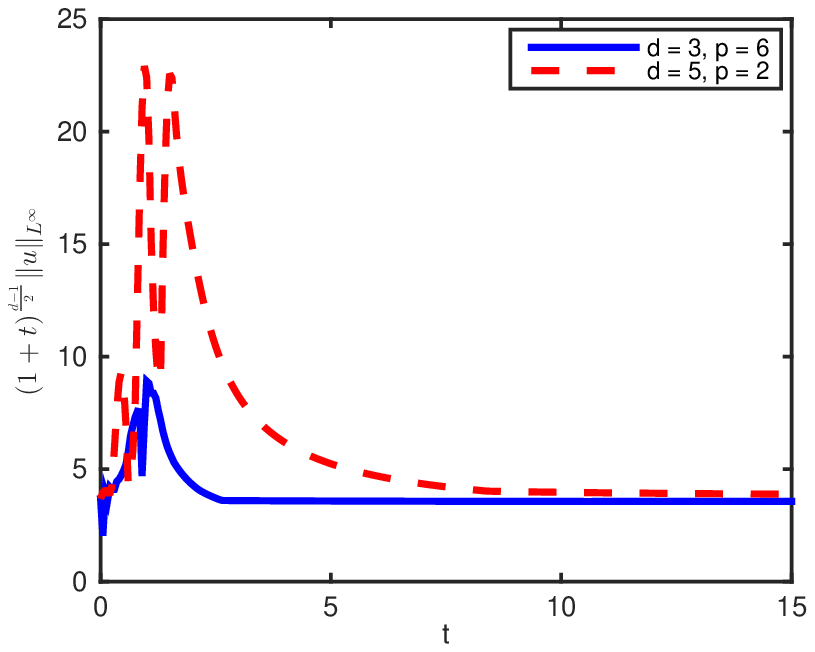}
\end{subfigure}
\caption{\footnotesize   Time evolution of Sobolev norms (first row), Besov norms (second row), and higher norms (last row) for NLW with initial condition (\ref{ring0}).}\label{Figure2norm}
\end{figure}

\clearpage
\subsection{Case 3. Incoming ring data.} In Case 3 we take
\begin{equation}\label{ring1}
u_0 = 10r^2\exp(-r^2),\qquad u_1 = \partial_r u_0 + \tfrac{d-2}{r}u_0, \qquad\mbox{for \ \ $r \ge 0$.}
\end{equation}
As described above, such initial data will lead to some initial `focusing' at the level of the linear wave equation.  This effect is countered by the defocusing nonlinearity.

{Figure \ref{Figure3u} illustrates the time evolution of the solution $u$, where the plot for $t = 0$ is the same as that in Figure \ref{Figure2u}. Due to the different initial velocity $u_1$, the evolution of solution in Cases 2 and 3 is initially quite different (cf. Figures \ref{Figure2u} and \ref{Figure3u} for $t = 1$), but after a long time the influence of initial velocity greatly decreases. }

\begin{figure}[!htbp]
\centering
\begin{subfigure}{.5\textwidth}
\centering
\includegraphics[width=1\linewidth]{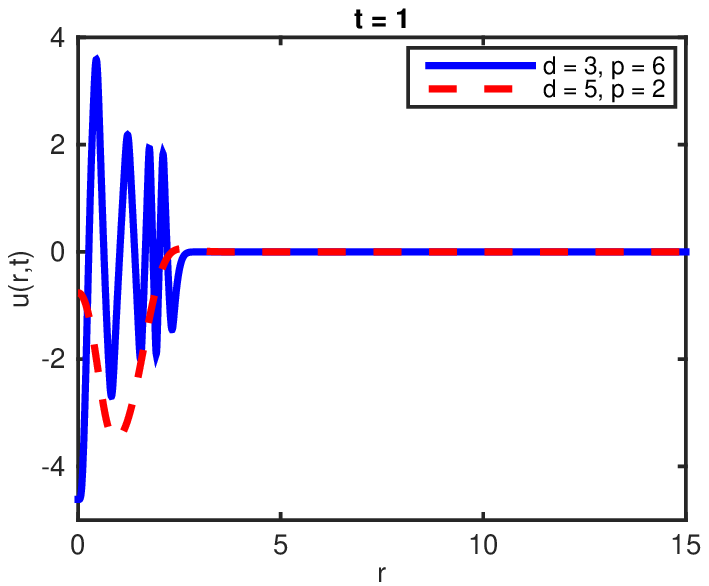}
\end{subfigure}%
\begin{subfigure}{.5\textwidth}
\centering
\includegraphics[width=1\linewidth]{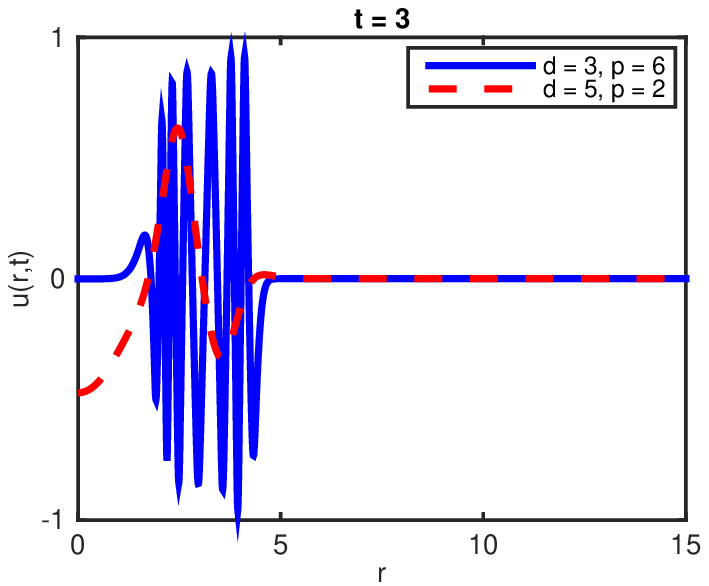}
\end{subfigure}
\begin{subfigure}{.5\textwidth}
\centering
\includegraphics[width=1\linewidth]{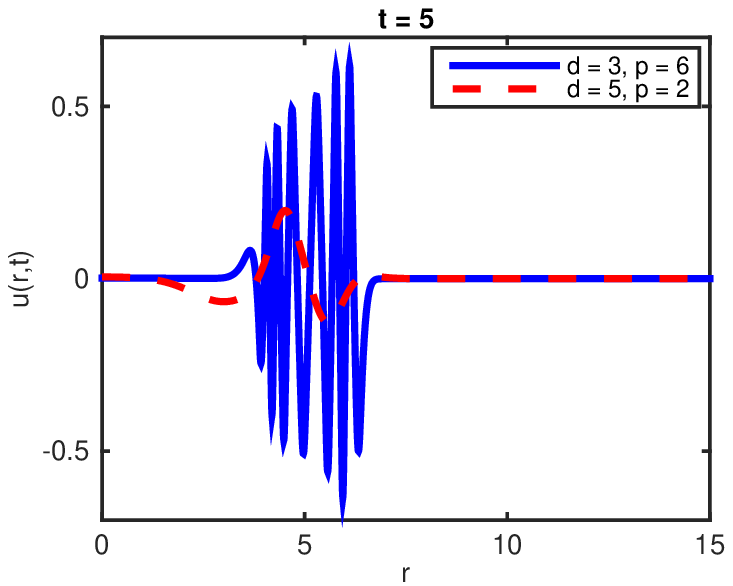}
\end{subfigure}%
\begin{subfigure}{.5\textwidth}
\centering
\includegraphics[width=1\linewidth]{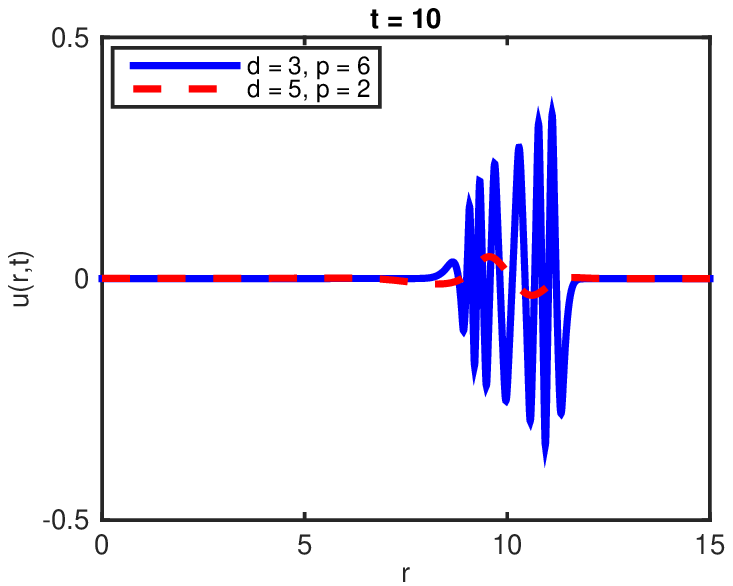}
\end{subfigure}
\caption{\footnotesize   Time evolution of the solution to NLW (\ref{NLW-ode}) with initial condition (\ref{ring1}).}\label{Figure3u}
\end{figure}

{The time evolution of the norms in Figure \ref{Figure3norm} is quite similar to the observations in Figure \ref{Figure2norm}.  This suggests that the initial velocity $u_1$ ultimately plays an insignificant role in the study of long-time behavior of solution.}

\begin{figure}[!htbp]
\centering
\begin{subfigure}{.5\textwidth}
\centering
\includegraphics[width=1\linewidth]{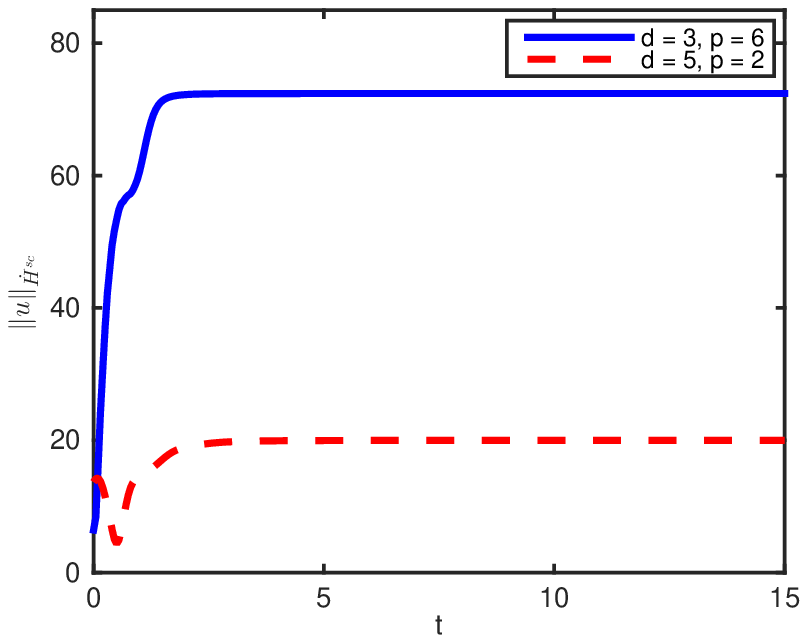}
\end{subfigure}%
\begin{subfigure}{.5\textwidth}
\centering
\includegraphics[width=1\linewidth]{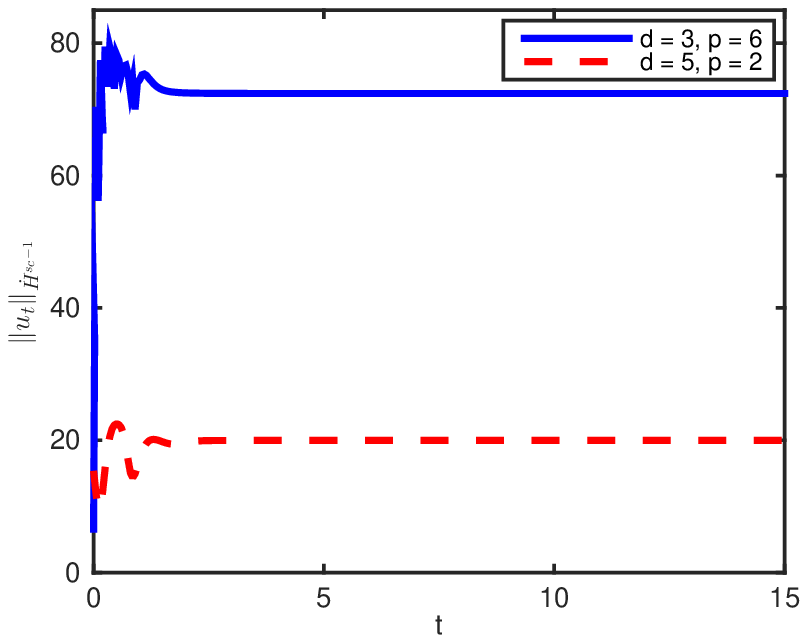}
\end{subfigure}
\begin{subfigure}{.5\textwidth}
\centering
\includegraphics[width=1\linewidth]{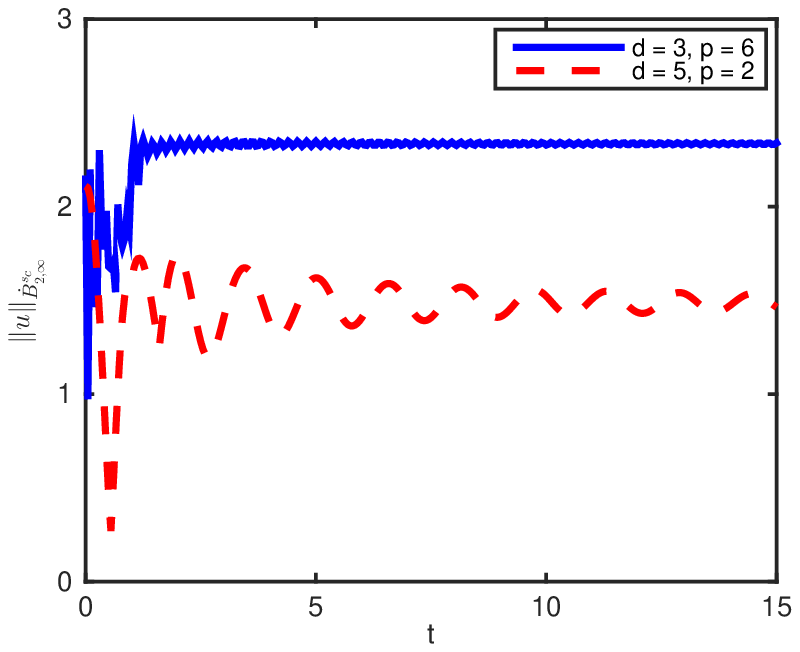}
\end{subfigure}%
\begin{subfigure}{.5\textwidth}
\centering
\includegraphics[width=1\linewidth]{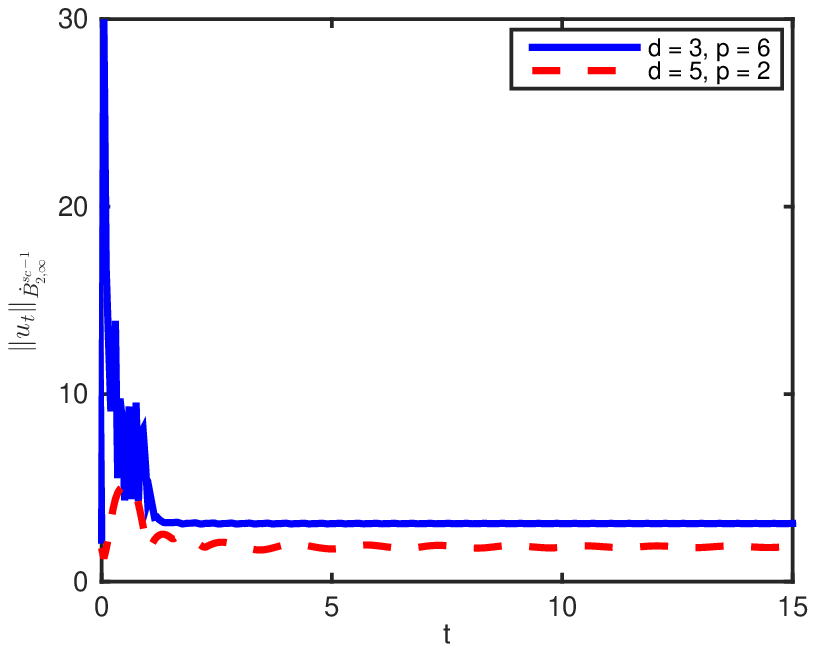}
\end{subfigure}
\begin{subfigure}{.5\textwidth}
\centering
\includegraphics[width=1\linewidth]{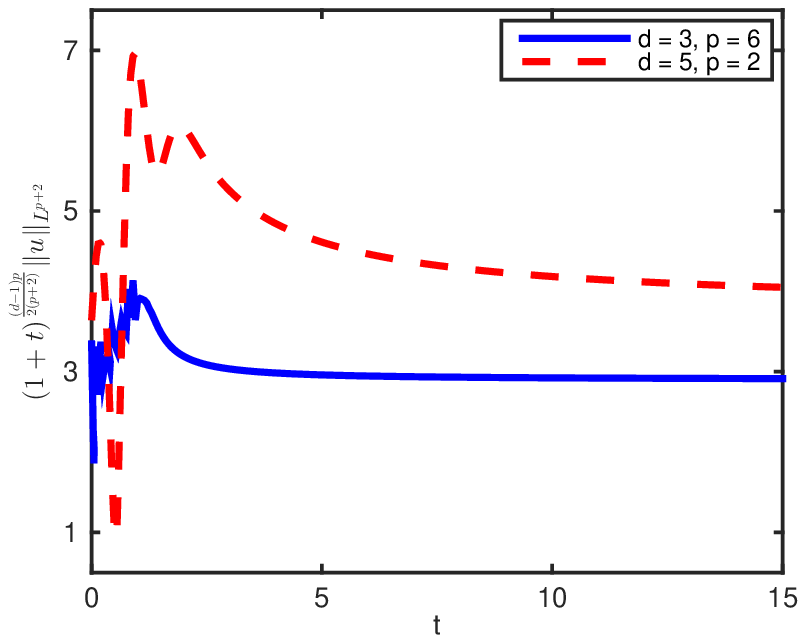}
\end{subfigure}%
\begin{subfigure}{.5\textwidth}
\centering
\includegraphics[width=1\linewidth]{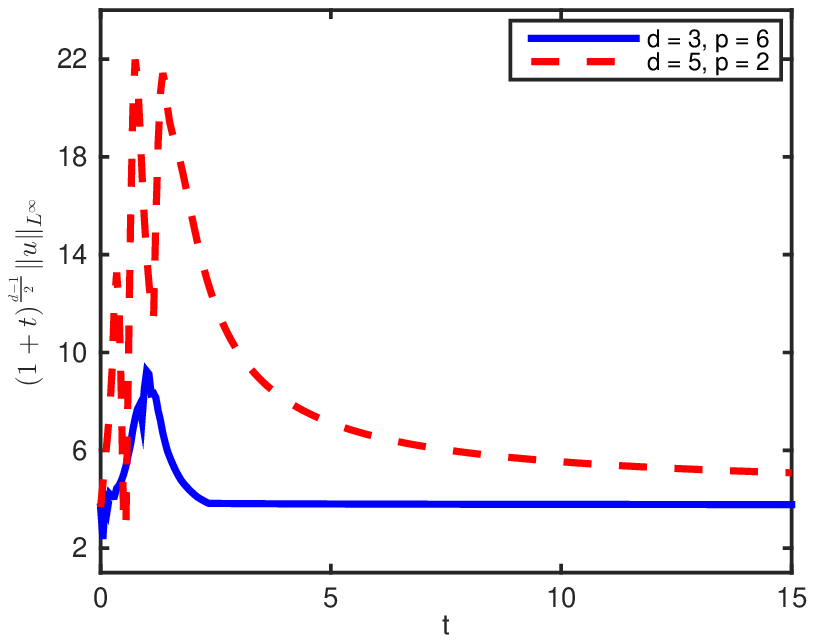}
\end{subfigure}
\caption{\footnotesize   Time evolution of the Sobolev norms (first row), Besov norms (second row), and higher norms (last row) for NLW with initial condition (\ref{ring1}).}\label{Figure3norm}
\end{figure}


\clearpage

\subsection{Case 4. Oscillatory Gaussian data} Next, we consider the initial data as follows:

{ \begin{equation}\label{gring1}
u_0 = 5\exp(-r^2)\sin(3r),\qquad u_1 = 0, \qquad\mbox{for \ \ $r \ge 0$.}
\end{equation}}

The numerical results are consistent with all previous cases: the pointwise decay and outward travel of the solution, with higher norms decaying at rates consistent with the linear wave equation.  Moreover, the Sobolev norms settle down, and the Besov norms become small compared to the Sobolev norms.

\begin{figure}[!htbp]
\centering
\begin{subfigure}{.5\textwidth}
\centering
\includegraphics[width=1\linewidth]{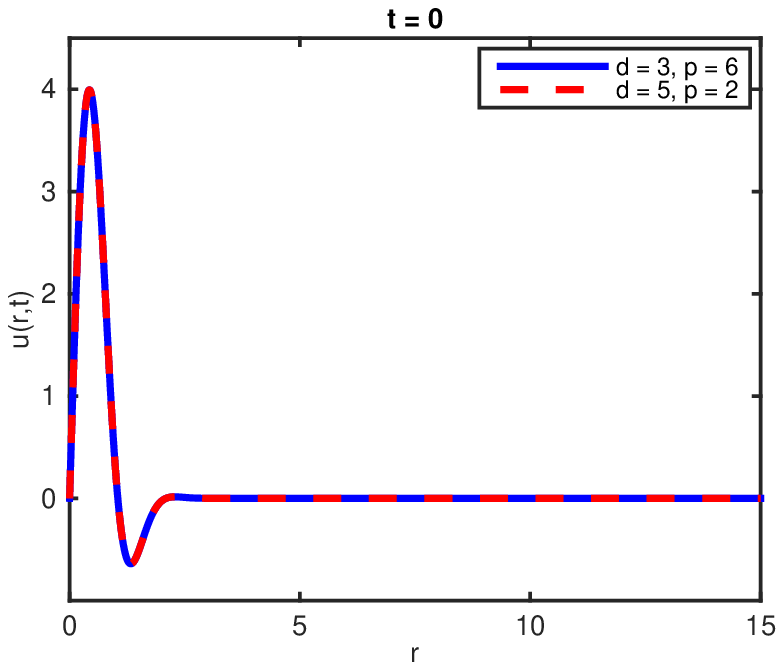}
\end{subfigure}%
\begin{subfigure}{.5\textwidth}
\centering
\includegraphics[width=1\linewidth]{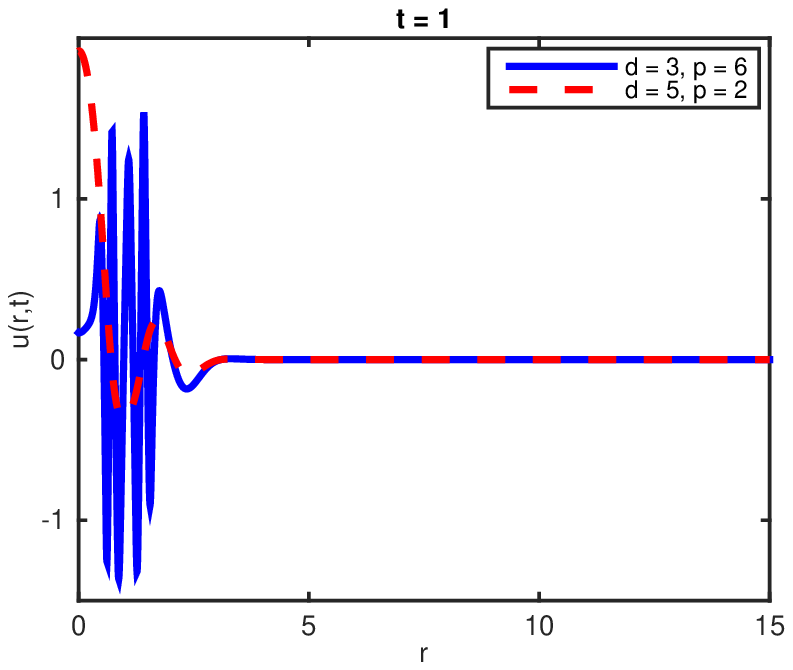}
\end{subfigure}
\begin{subfigure}{.5\textwidth}
\centering
\includegraphics[width=1\linewidth]{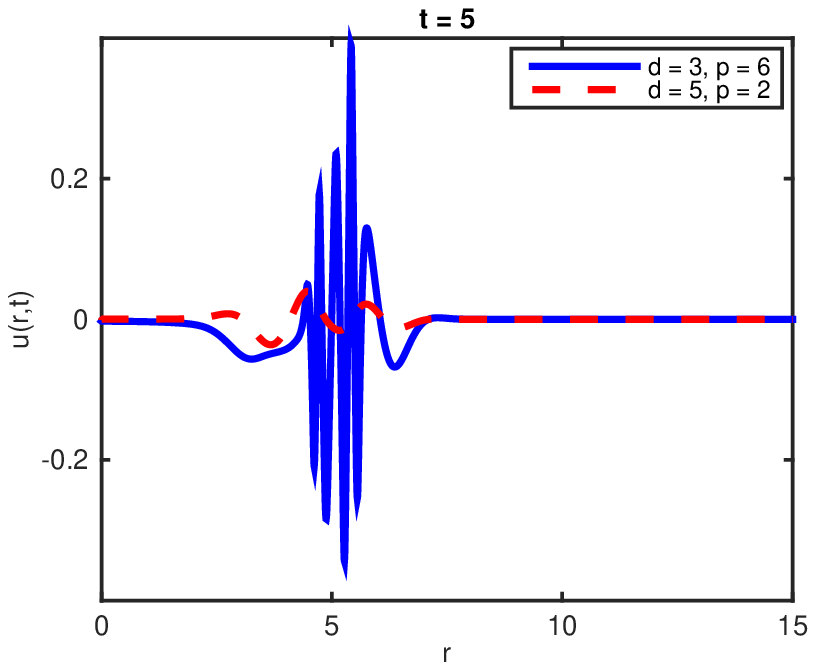}
\end{subfigure}%
\begin{subfigure}{.5\textwidth}
\centering
\includegraphics[width=1\linewidth]{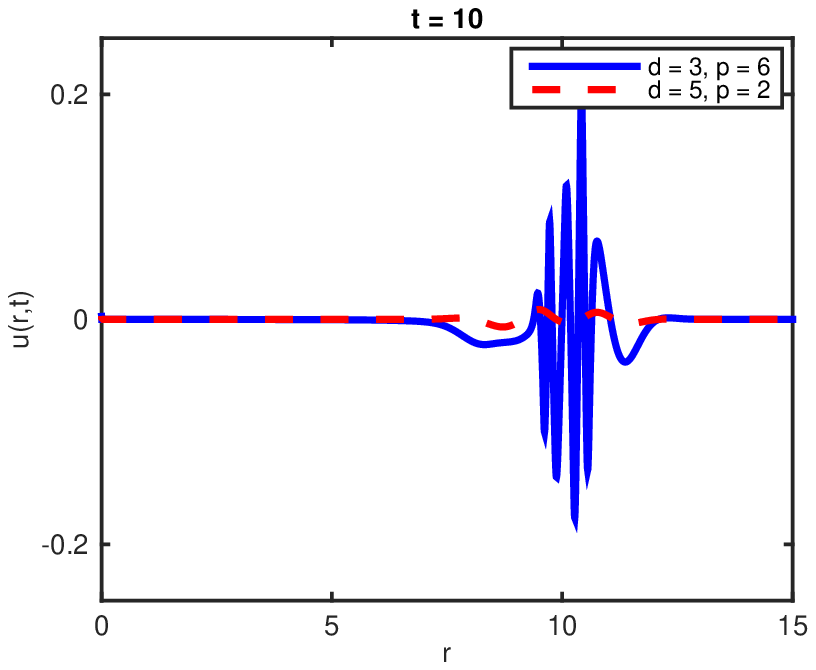}
\end{subfigure}
\caption{\footnotesize   Time evolution of solution of NLW (\ref{NLW-ode}) with initial condition (\ref{gring1}).}\label{Figure4u}
\end{figure}

\begin{figure}[!htbp]
\centering
\begin{subfigure}{.5\textwidth}
\centering
\includegraphics[width=1\linewidth]{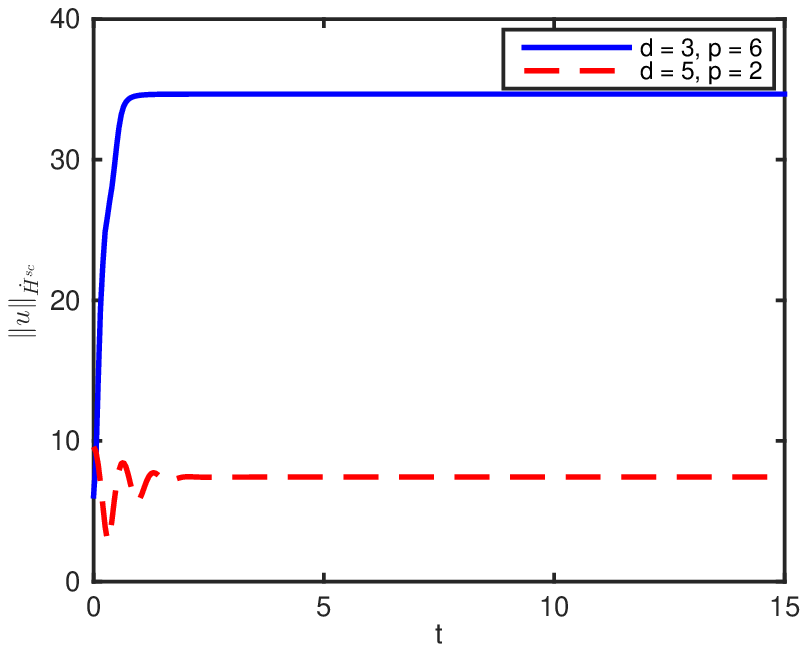}
\end{subfigure}%
\begin{subfigure}{.5\textwidth}
\centering
\includegraphics[width=1\linewidth]{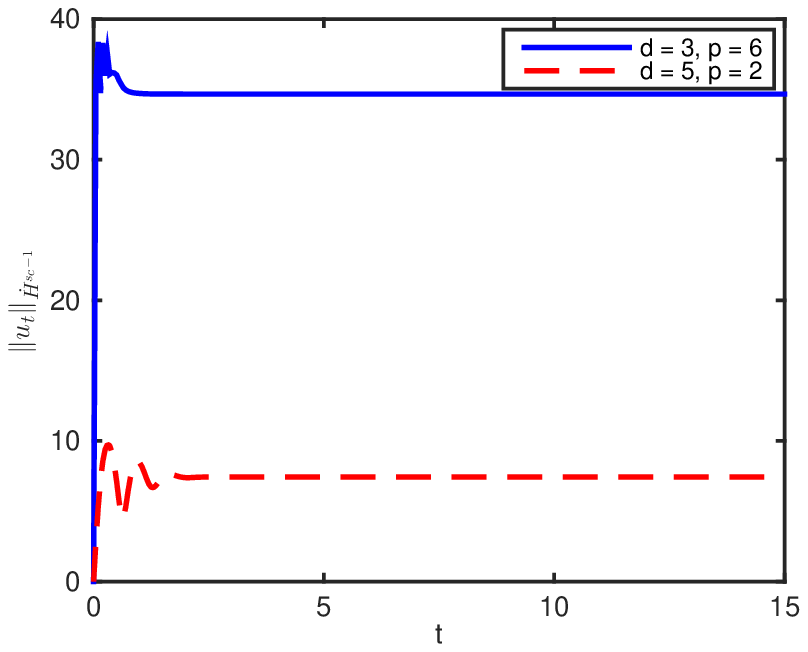}
\end{subfigure}
\begin{subfigure}{.5\textwidth}
\centering
\includegraphics[width=1\linewidth]{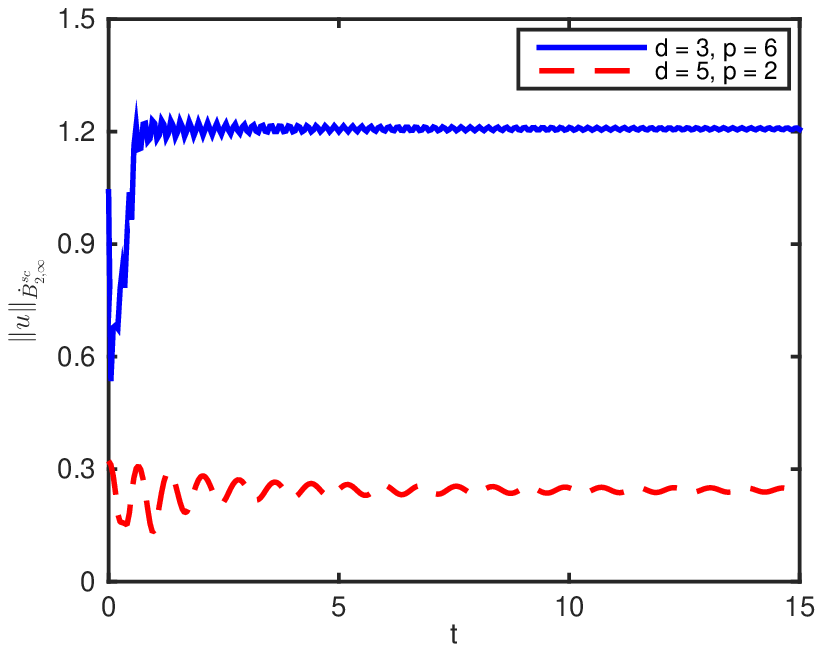}
\end{subfigure}%
\begin{subfigure}{.5\textwidth}
\centering
\includegraphics[width=1\linewidth]{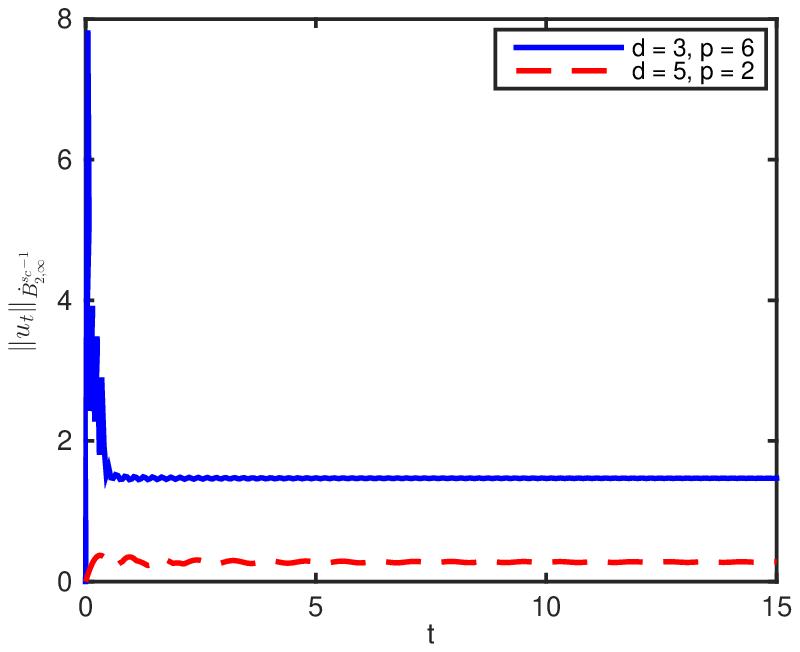}
\end{subfigure}
\begin{subfigure}{.5\textwidth}
\centering
\includegraphics[width=1\linewidth]{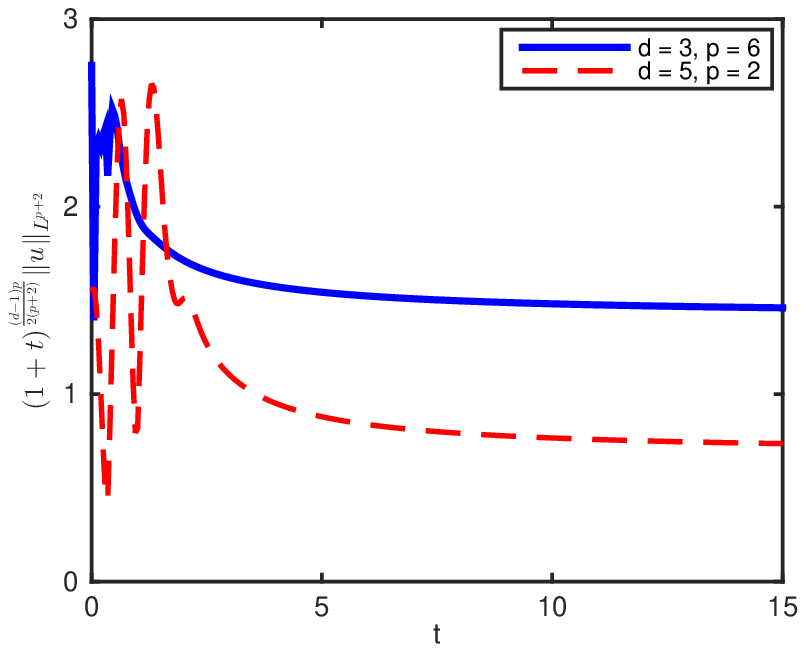}
\end{subfigure}%
\begin{subfigure}{.5\textwidth}
\centering
\includegraphics[width=1\linewidth]{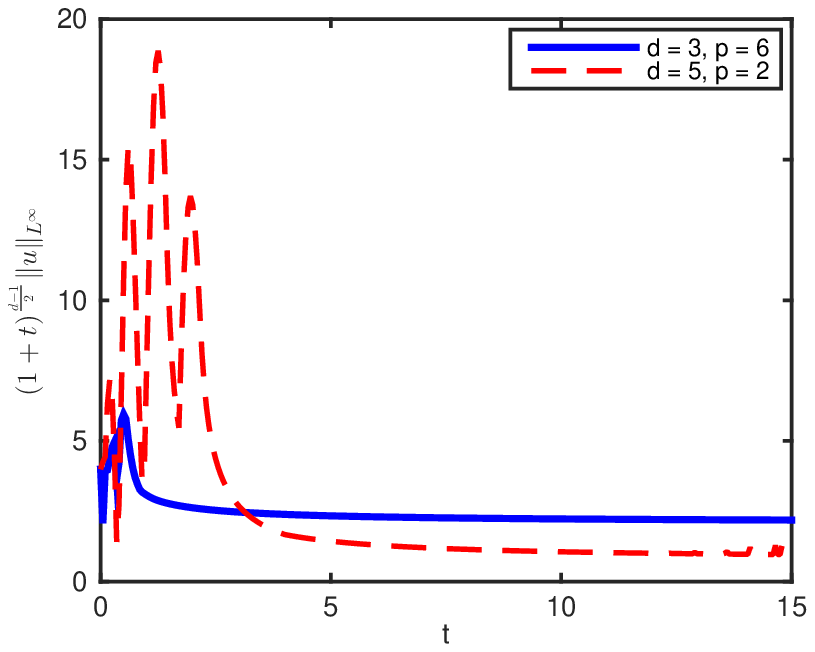}
\end{subfigure}
\caption{\footnotesize   Time evolution of the Sobolev norms (first row), Besov norms (second row), and higher norms (last row) for NLW with initial condition (\ref{gring1}).}\label{Figure4norm}
\end{figure}


\clearpage

\subsection{Case 5. Incoming oscillatory Gaussian data} Finally, we consider the same initial position as $u_0$, but with incoming data: 
{ \begin{equation}\label{ring2}
u_0 = 5\exp(-r^2)\sin(3r),\qquad u_1 = \partial_r u_0 + \tfrac{d-2}{r}u_0, \qquad\mbox{for \ \ $r \ge 0$.}
\end{equation}}

Once again, we find that making the data incoming ultimately has little effect on the long-time behavior of solutions, and the numerical results are consistent with all that we have seen above.

\begin{figure}[!htbp]
\centering
\begin{subfigure}{.5\textwidth}
\centering
\includegraphics[width=1\linewidth]{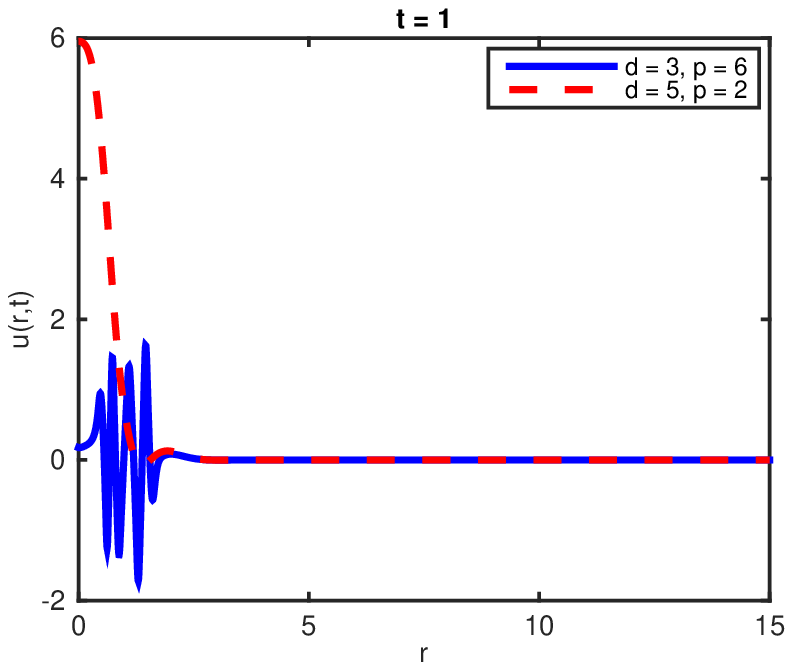}
\end{subfigure}%
\begin{subfigure}{.5\textwidth}
\centering
\includegraphics[width=1\linewidth]{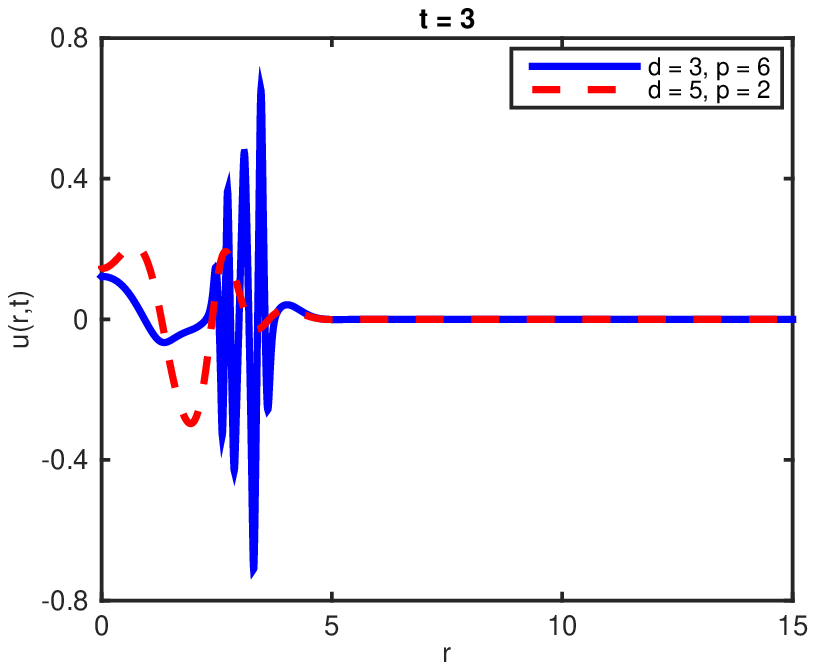}
\end{subfigure}
\begin{subfigure}{.5\textwidth}
\centering
\includegraphics[width=1\linewidth]{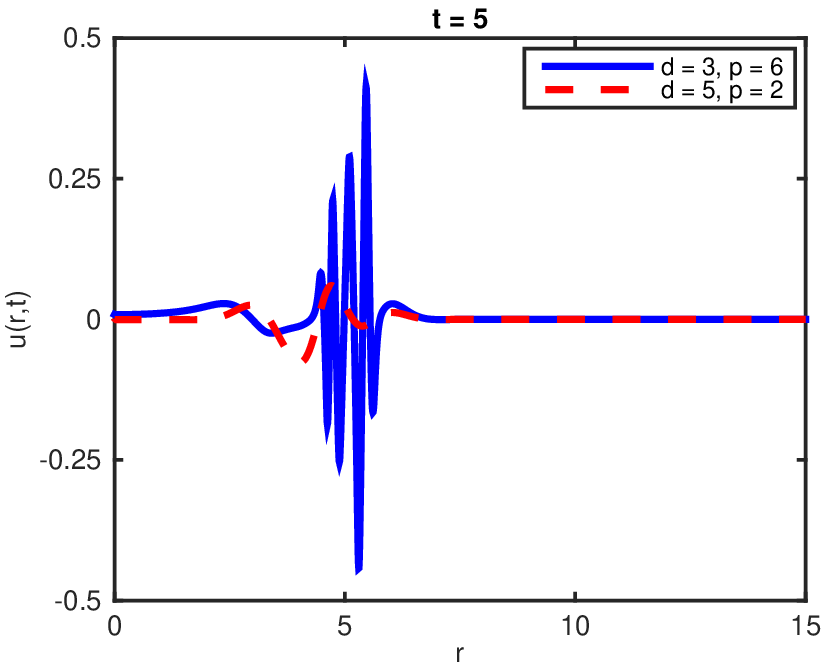}
\end{subfigure}%
\begin{subfigure}{.5\textwidth}
\centering
\includegraphics[width=1\linewidth]{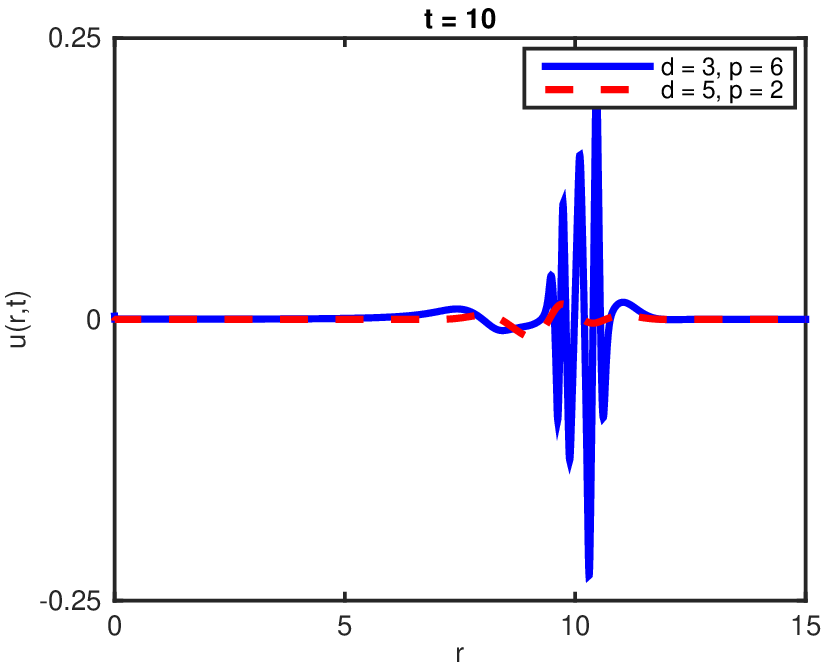}
\end{subfigure}
\caption{\footnotesize   Time evolution of solution of NLW (\ref{NLW-ode}) with initial condition (\ref{ring2}).}\label{Figure4u}
\end{figure}

\begin{figure}[!htbp]
\centering
\begin{subfigure}{.5\textwidth}
\centering
\includegraphics[width=1\linewidth]{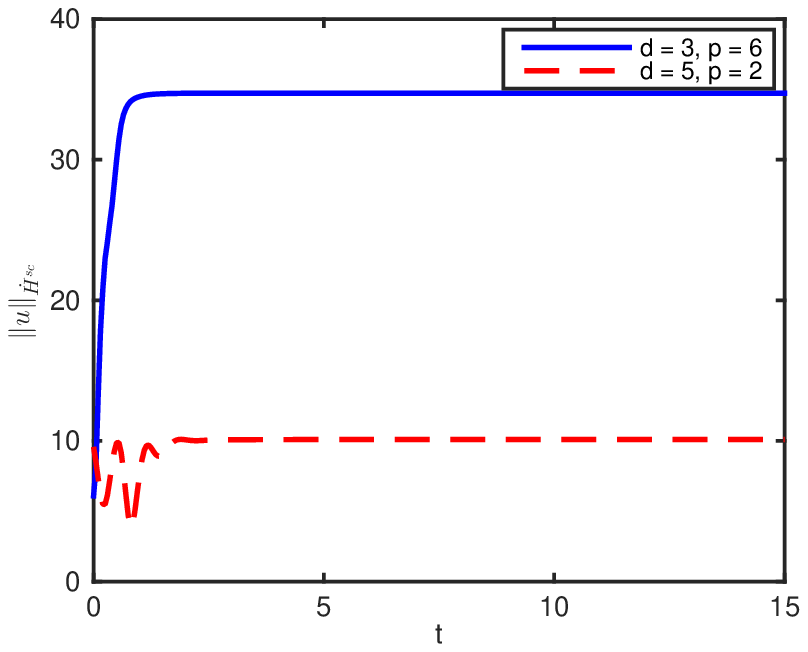}
\end{subfigure}%
\begin{subfigure}{.5\textwidth}
\centering
\includegraphics[width=1\linewidth]{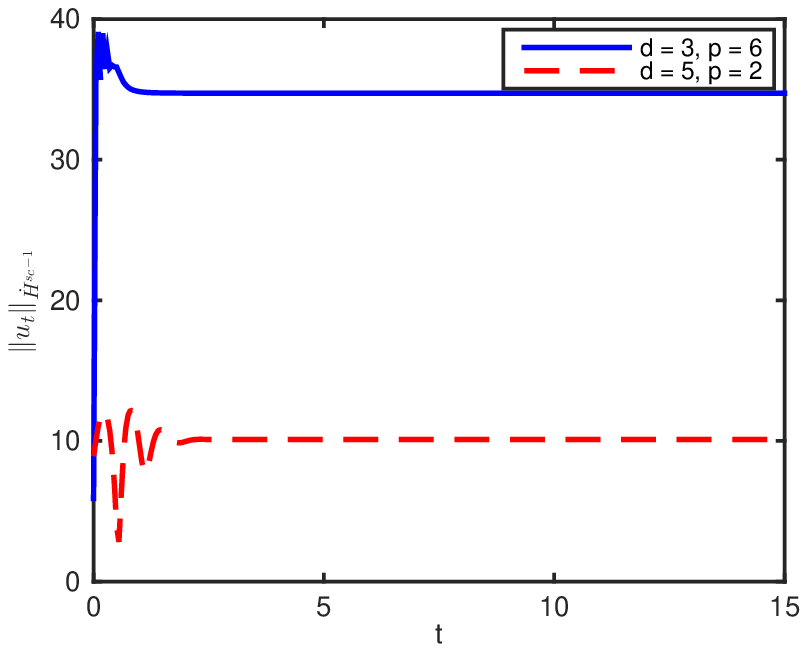}
\end{subfigure}
\begin{subfigure}{.5\textwidth}
\centering
\includegraphics[width=1\linewidth]{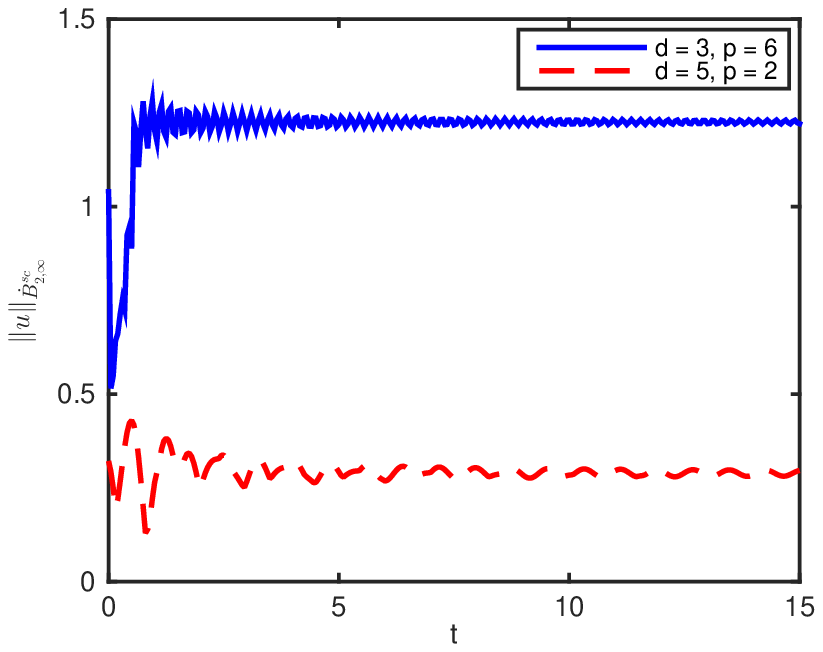}
\end{subfigure}%
\begin{subfigure}{.5\textwidth}
\centering
\includegraphics[width=1\linewidth]{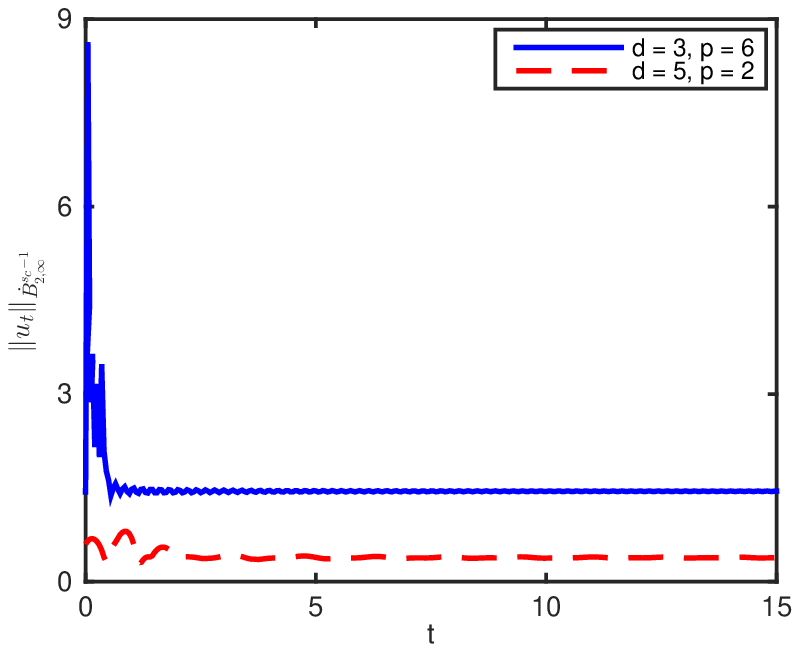}
\end{subfigure}
\begin{subfigure}{.5\textwidth}
\centering
\includegraphics[width=1\linewidth]{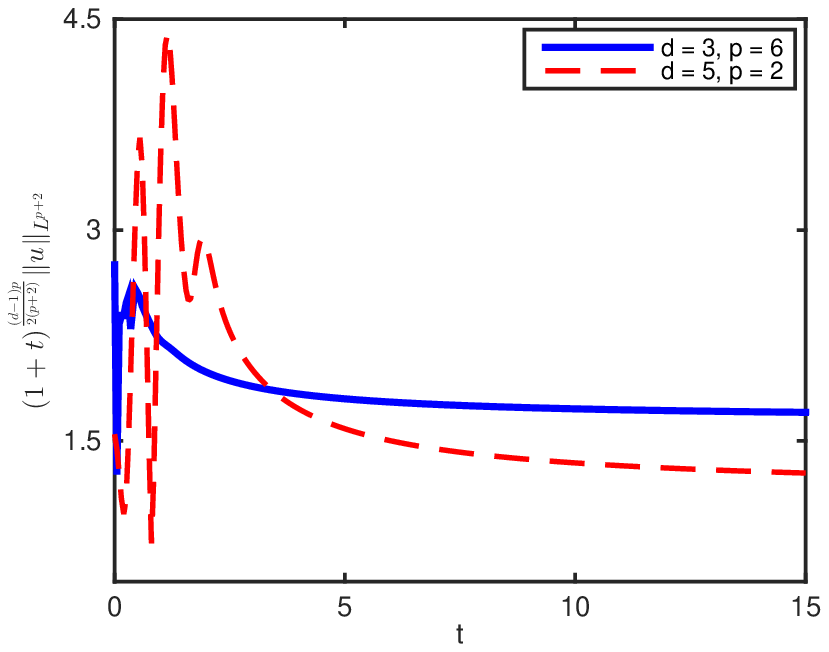}
\end{subfigure}%
\begin{subfigure}{.5\textwidth}
\centering
\includegraphics[width=1\linewidth]{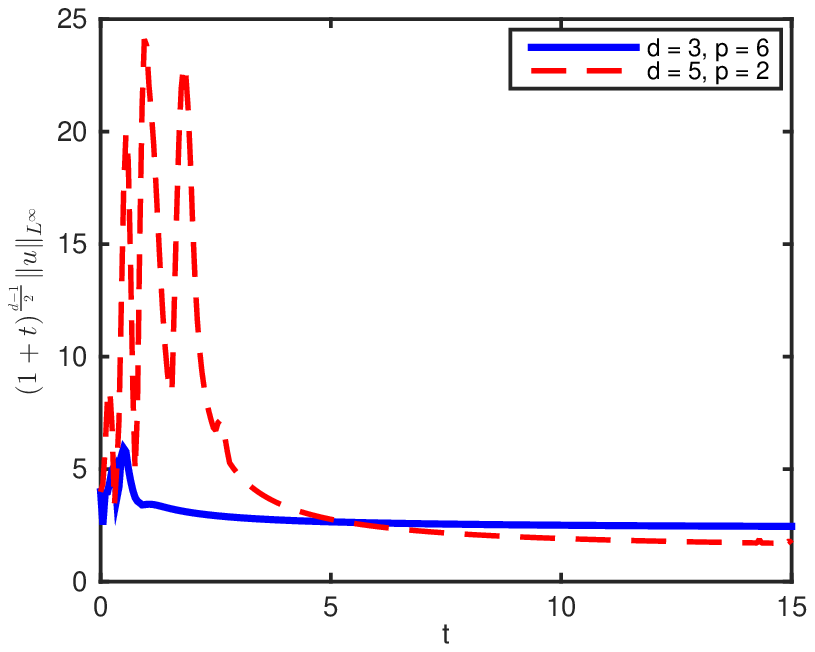}
\end{subfigure}
\caption{\footnotesize   Time evolution of the Sobolev norms (first row), Besov norms (second row), and higher norms (last row) for NLW with initial condition (\ref{ring2}).}\label{Figure5norm}
\end{figure}

\clearpage

\section{A simple scattering result}\label{S:Scattering}

In this section, we demonstrate that if a solution to \eqref{nlw} is sufficiently dispersed in frequency compared to its critical Sobolev norm, then it scatters.  We present this result because the numerical results presented above demonstrate that this phenomenon occurs, at least for the sets of initial conditions that we considered.  We note, however, that this result is still fundamentally a small-data result.  It would require some significant new input to rigorously demonstrate that such dispersion in frequency occurs in general. 

For the sake of concreteness, we focus on the particular case $(d,p)=(3,6)$ considered above, in which case $s_c=\frac76$. 

\begin{proposition}\label{P:A1} Suppose $\|(u_0,u_1)\|_{\dot H^{\frac76}\times\dot H^{\frac16}}=E$.  Then there exists $\eta_0=\eta_0(E)>0$ such that if
\[
\|(u_0,u_1)\|_{\dot B^{\frac76}_{2,\infty}\times\dot B^{\frac16}_{2,\infty}}<\eta<\eta_0,
\]
then the solution to $-\partial_t^2 u + \Delta u = |u|^{6}u$ with data $(u_0,u_1)$ is global in time and obeys the space-time bounds
\[
\|u\|_{L_{t,x}^{12}(\R\times\R^3)}+ \||\nabla|^{\frac23} u\|_{L_{t,x}^4(\R\times\R^3)}<\infty.
\]
In particular, $u$ scatters. 
\end{proposition}

To prove this result, it is useful to complexify the equation and introduce the variable
\[
v=u-i|\nabla|^{-1}\partial_t u,\qtq{where}|\nabla|^{-1}=(-\Delta)^{-\frac12}. 
\]
Then $(u,\partial_t u)\in \dot H^{s_c}\times \dot H^{s_c-1}$ if and only if $v\in \dot H^{s_c}$ (with comparable norms), with a similar statement concerning the Besov norms. The equation \eqref{nlw} is then equivalent to
\begin{equation}\label{nlw-v}
i\partial_t v =  - |\nabla| v - |\nabla|^{-1}(|\Re v|^6\Re v) =0,
\end{equation}
which in turn is equivalent to the Duhamel formulation
\begin{equation}\label{Duhamel}
v(t) = e^{it|\nabla|}v_0 + i\int_0^t e^{i(t-s)|\nabla|}|\nabla|^{-1}(|\Re v|^6 \Re v)\,ds. 
\end{equation}

The proof of Proposition~\ref{P:A1} will rely primarily on Strichartz estimates for the operator $e^{it|\nabla|}$.  To begin, we have the following estimates (see, for example,  \cite{V-Oberwolfach}): Let $(q,r)$ and $(\tilde q,\tilde r)$ be wave admissible in $3d$, i.e.
\[
\tfrac{1}{q}+\tfrac{1}{r}\leq \tfrac{1}{2},\quad (q,r)\neq(2,\infty). 
\]
Defining 
\[
\gamma(q,r) = \tfrac{3}{2}-\tfrac{1}{q}-\tfrac{3}{r},
\] 
we have: 
\begin{align}
\| e^{it|\nabla|}f\|_{L_t^q L_x^r} & \lesssim \| |\nabla|^{\gamma(q,r)} f\|_{L^2}, \label{Str1} \\
\biggl\| \int_0^t e^{i(t-s)|\nabla|}F(s)\,ds\bigg\|_{L_t^q L_x^r} & \lesssim \| |\nabla|^{\gamma(q,r)+\gamma(\tilde q,\tilde r)}F\|_{L_t^{\tilde q'} L_x^{\tilde r'}}.\label{Str2}
\end{align}

The specific space-time norms appearing Proposition~\ref{P:A1} are critical norms for the case $(d,p)=(3,6)$.  That is, they are invariant under the rescaling \eqref{scaling}.  

We define
\begin{equation}\label{X}
\|e^{it|\nabla|}f\|_{X(I)} := \|e^{it|\nabla|}f\|_{L_{t,x}^{12}(I\times\R^3)} + \| |\nabla|^{\frac23} e^{it|\nabla|}f\|_{L_{t,x}^{4}(I\times\R^3)}.
\end{equation}
By the Strichartz estimates above (noting that $\gamma(4,4)=\tfrac12$), we have
\[
\|e^{it|\nabla|}f\|_{X(\R)} \lesssim \| |\nabla|^{\frac76} f\|_{L^2(\R^3)}. 
\]

A less standard ingredient for the proof of Proposition~\ref{P:A1} will be the following refined Strichartz estimate for this norm.

\begin{lemma}[Refined Strichartz estimate]\label{L:refined} There exists $\theta\in(0,1)$ such that
\[
\|e^{it|\nabla|}f\|_{X(\R)} \lesssim \||f\|_{\dot B_{2,\infty}^{s_c}}^\theta \| |\nabla|^{\frac76} f\|_{L^2}^{1-\theta}.
\]
\end{lemma}

\begin{proof} As the Besov norm is controlled by the Sobolev norm, it suffices to prove an estimate of this form for each norm appearing in \eqref{X}.  Let us focus on the $L_{t,x}^4$ norm, as the proof for the $L_{t,x}^{12}$ norm follows along similar lines.  

We employ the Littlewood--Paley frequency decomposition
\[
f=\sum_{N\in2^{\mathbb{Z}}} P_N f,
\]
as described in Section~\ref{S:Notation}. Let us denote
\[
u(t) = e^{it|\nabla|} |\nabla|^{\frac23} f,\quad u_N(t) = e^{it|\nabla|}|\nabla|^{\frac23} P_N f. 
\]
By the Littlewood--Paley square function estimate, we may write
\begin{align*}
\|u\|_{L_{t,x}^4}^4 & \sim \iint \biggl[\sum_{N\in 2^{\mathbb{Z}}} |u_N(t,x)|^2\biggr]^2\,dx\,dt \\
& \lesssim \iint \sum_{N_1\leq N_2} |u_{N_1}|^2|u_{N_2}|^2\,dx\,dt. 
\end{align*}

Now let $\eps>0$ be a small parameter and define two pairs of sharp wave-admissible exponents $(q_l,r_l)$ and $(q_h,r_h)$ by
\[
(\tfrac{1}{q_l},\tfrac{1}{r_l}) = (\tfrac14+\eps,\tfrac14-\eps),\quad (\tfrac{1}{q_h},\tfrac{1}{r_h})=(\tfrac14-\eps,\tfrac14+\eps). 
\]
Note that
\[
\gamma(q_l,r_l) = \tfrac12+2\eps,\quad \gamma(q_h,r_h)=\tfrac12-2\eps. 
\]
In particular, by \eqref{Str1} and Bernstein inequalities\footnote{Bernstein inequalities refer to the following general estimates for frequency-localized functions:
\begin{align*}
\|P_N f\|_{L^{r_2}(\R^d)}&\lesssim N^{\frac{d}{r_2}-\frac{d}{r_1}}\|P_N f\|_{L^{r_1}(\R^d)},\quad 1\leq r_1\leq r_2\leq\infty, \\
\|  |\nabla|^sP_N f\|_{L^r(\R^d)}&  \sim N^s\| P_N f\|_{L^r(\R^d)}, \quad 1<r<\infty,\quad s\in\R. 
\end{align*}}, we have
\begin{align*}
\| u_{N_1}\|_{L_t^{q_l} L_x^{r_l}} & \lesssim \| |\nabla|^{\frac76+2\eps} f_{N_1}\|_{L^2} \lesssim N_1^{2\eps}\|f_{N_1}\|_{\dot H^{\frac76}}, \\
\| u_{N_2}\|_{L_t^{q_h} L_x^{r_h}} & \lesssim \| |\nabla|^{\frac76-2\eps} f_{N_2}\|_{L^2} \lesssim N_2^{-2\eps}\|f_{N_2}\|_{\dot H^{\frac76}}. 
\end{align*}
Thus, continuing from above and applying \eqref{Str1} and Cauchy--Schwarz, we have
\begin{align*}
\|u\|_{L_{t,x}^4}^4 & \lesssim \sum_{N_1\leq N_2} \|u_{N_1}\|_{L_t^{q_l}L_x^{r_l}}\|u_{N_1}\|_{L_{t,x}^4} \|u_{N_2}\|_{L_{t,x}^4}\|u_{N_2}\|_{L_t^{q_h}L_x^{r_h}} \\
& \lesssim \bigl[\sup_{N} \|u_{N}\|_{L_{t,x}^4}\bigr]^2 \sum_{N_1\leq N_2}\bigl(\tfrac{N_1}{N_2}\bigr)^{2\eps}\|f_{N_1}\|_{\dot H^{\frac76}}\|f_{N_2}\|_{\dot H^{\frac76}} \\
& \lesssim \bigl[\sup_N \|f_N\|_{\dot H^{\frac76}}\bigr]^2 \|f\|_{\dot H^{\frac76}}^2,
\end{align*}
which yields the desired estimate.

In the case of the $L_{t,x}^{12}$ norm, the situation is actually simpler because we work away from the `strictly admissible' line $\tfrac{1}{q}+\tfrac{1}{r}=\tfrac{1}{2}$.  In this case we put one of the lowest frequency pieces $u_{N_1}$ in $L_t^{12}L_x^{12+}$ and one of the highest frequency pieces $u_{N_{6}}$ in $L_t^{12} L_x^{12-}$.  All of the remaining pieces are taken out with a supremum in $N$ of the $L_{t,x}^{12}$ norm, which yields the Besov norm after an application of Strichartz (as above). Using Bernstein and Strichartz for the low and high frequency pieces, we get a gain of $(\tfrac{N_1}{N_{6}})^{\eps}$, which can be used to defeat the logarithmic loss from the sums in $N_2,\dots,N_5$ and to sum in $N_1,N_6$ via Cauchy--Schwarz, just as above. \end{proof}

We turn to the proof of Proposition~\ref{P:A1}.

\begin{proof}[Proof of Proposition~\ref{P:A1}] The standard well-posedness theory for \eqref{nlw} yields a local-in-time solution to \eqref{nlw-v} satisfying the Duhamel formula \eqref{Duhamel}.  It will therefore suffice to prove the estimate 
\begin{equation}\label{bootstrap}
\|v\|_{X(I)} \lesssim \|v_0\|_{\dot B^{\frac76}_{2,\infty}}^\theta \|v_0\|_{\dot H^{\frac76}}^{1-\theta} + \|v\|_{X(I)}^7
\end{equation}
on any interval $I\subset\R$ containing $t=0$, where $\theta\in(0,1)$ is as in Lemma~\ref{L:refined}. Indeed, choosing $\eta_0=\eta_0(E)$ sufficiently small (where $E,\eta_0$ are as in the statement of Proposition~\ref{P:A1}), we may make the first term on the right-hand side of \eqref{bootstrap} as small as we wish. Then, by a standard continuity argument, \eqref{bootstrap} implies
\[
\|v\|_{X(\R)}\lesssim  \|v_0\|_{\dot B^{\frac76}_{2,\infty}}^\theta \|v_0\|_{\dot H^{\frac76}}^{1-\theta},
\]
giving the desired global space-time bounds for $v$.

It therefore remains to prove \eqref{bootstrap}.  Recalling the Duhamel formula \eqref{Duhamel}, we apply Lemma~\ref{L:refined}, \eqref{Str2}, and the fractional chain rule\footnote{This refers to the estimate \[ \||\nabla|^s F(u)\|_{L^r} \lesssim \|F'(u)\|_{L^{r_1}}\| |\nabla|^s u\|_{L^{r_2}}\] for $0<s<1$ and $1<r,r_1,r_2<\infty$ satisfying $\tfrac{1}{r}=\tfrac{1}{r_1}+\tfrac{1}{r_2}$.} to estimate
\begin{align*}
\|v\|_X & \lesssim \|v_0\|_{\dot B^{\frac76}_{2,\infty}}^\theta \|v_0\|_{\dot H^{\frac76}}^{1-\theta} + \| |\nabla|^{\frac23}\bigl(|u|^6 u\bigr)\|_{L_{t,x}^{\frac43}} \\
& \lesssim \|v_0\|_{\dot B^{\frac76}_{2,\infty}}^\theta \|v_0\|_{\dot H^{\frac76}}^{1-\theta} + \|v\|_{L_{t,x}^{12}}^6 \| |\nabla|^{\frac23} v \|_{L_{t,x}^4} \\
& \lesssim \|v_0\|_{\dot B^{\frac76}_{2,\infty}}^\theta \|v_0\|_{\dot H^{\frac76}}^{1-\theta}+ \|v\|_{X}^7,
\end{align*}
giving \eqref{bootstrap}, as desired. \end{proof}

\appendix

\section{Incoming waves for the radial wave equation}\label{Incoming}

In this section we describe the notion of incoming/outgoing waves for the radial wave equation.  In particular, we wish to explain the origin of the `incoming' condition
\begin{equation}\label{appendix-incoming}
u_1 = \partial_r u_0 + \tfrac{d-2}{r}u_0.
\end{equation}  
Our discussion will be brief---for more details, see \cite{Marius}.  

First, consider the radial wave equation in three space dimensions: 
\begin{equation}\label{linear-3d}
u_{tt}-u_{rr} - \tfrac{2}{r}u_r = 0. 
\end{equation}
If $u$ and $v$ are related through
\[
u(r) = \tfrac{1}{r}\int_0^r v(\rho)\,d\rho,
\]
then one finds that $u$ solving \eqref{linear-3d} is equivalent to $v$ solving the $1d$ wave equation 
\[
v_{tt}-v_{rr} =0;
\]
however, we should now view $v$ as a solution on the half-line $r\in(0,\infty)$ with the Neumann boundary condition $\partial_r v|_{r=0}=0$. 

If instead $u$ solves the radial equation in five space dimensions, i.e.
\begin{equation}\label{linear-5d}
u_{tt}-u_{rr}-\tfrac{4}{r}u_r =0,
\end{equation}
then (after a short computation) one finds that the same situation arises when $u$ and $v$ are related via
\[
u(r) = \tfrac{1}{r^3}\int_0^r (r^2-\rho^2)v(\rho)\,d\rho. 
\]

Now, for the $1d$ wave equation (with Neumann boundary conditions) one can use the explicit formula for solutions to see that every solution may be decomposed into an incoming and outgoing piece, i.e.
\[
v(t)=v_{in}(t) + v_{out}(t),
\]
where $v_{in}$ moves toward $r=0$ and $v_{out}$ moves toward $r=\infty$ (both with speed one). As $t\to\infty$, the solution becomes increasingly outgoing, while as $t\to-\infty$ the solution becomes increasingly incoming. Writing $(v_0,v_1)$ for the initial conditions of $v(t)$, one finds that at time zero the incoming component is
\[
v_- = \tfrac12(v_0+\partial_r^{-1} v_1),
\]
while the outgoing component is
\[
v_+ = \tfrac12(v_0-\partial_r^{-1} v_1),
\]
where $\partial_r^{-1}$ denotes the antiderivative.  For more detail, see \cite{Marius}.

We are interested in prescribing `incoming' initial data.  In this case, at the linear level there is an initial `focusing' of the solution toward the origin, which we can then observe is countered by the defocusing effect of the nonlinearity.  In particular, an initial data pair for $v$ is incoming if $v_+=0$, that is, if
\[
(v,\partial_t v)|_{t=0}=(v_0,\partial_r v_0).
\]
Let us now derive equivalent conditions at the level of $(u_0,u_1)$.  These are the conditions appearing in our choices of initial data. 

For the three-dimensional case, we have the simple relation $v=\partial_r(ru)$, and the incoming condition becomes
\begin{equation}\label{incoming-3d}
u_1 =\partial_r u_0 + \tfrac{1}{r}u_0.
\end{equation}
For the five-dimensional case, we instead get $\partial_r^3(r^3 u)=2v+\partial_r v$, and the incoming condition becomes
\begin{equation}\label{incoming-5d}
u_1 = \partial_r u_0 + \tfrac{3}{r}u_0.
\end{equation}
In general, one derives the condition \eqref{appendix-incoming}.

We close this section with the following lemma, which shows that for sufficiently regular $u_0$, the initial velocity velocity $u_1$ still belongs to $L^2 \cap \dot H^{s}$ for suitable $s$ (despite the presence of the singular term $1/r$). 
\begin{lemma}\label{Lemma} Let $u$ be a Schwartz function in $\R^d$, with $d\geq 3$.  Then $\tfrac{1}{|x|}u$ belongs to $\dot H^s$ for any $0\leq s\leq\min\{\tfrac{d}{2}-1,1\}$.  In fact, 
\begin{equation}\label{infact}
\|\tfrac{1}{|x|}u\|_{\dot H^{s}}\lesssim \|u\|_{\dot H^{s+1}}
\end{equation}
for any such $s$.
\end{lemma}

\begin{proof} It is enough to prove \eqref{infact}. This estimate is very similar to Hardy's inequality, which states
\[
\| |x|^{-s}u\|_{L^p(\R^d)}\lesssim \| |\nabla|^s u\|_{L^p(\R^d)}\qtq{for} 1<p<\tfrac{d}{s}
\]
(see e.g. \cite{TaoDispersive}).  In particular, an application of Hardy's inequality (and the boundedness of Riesz potentials) reduces \eqref{infact} to the commutator estimate
\[
\| [\tfrac1{|x|},|\nabla|^s] u\|_{L^2} \lesssim \| |\nabla|^{s+1}u\|_{L^2}. 
\] 
For this we will use the Fourier transform.  The key fact that we need is
\[
\F |x|^{-\alpha} \sim |\xi|^{\alpha-d} \qtq{for} 0<\alpha<d,
\]
which can be deduced using scaling and symmetry properties of the Fourier transform, or by a direct computation using the Gamma function (see \cite[Lemma~1, p.117]{SteinSingular}).  In particular, we are faced with estimating
\[
\|[\tfrac{1}{|x|},|\nabla|^s]u\|_{L^2} \sim \biggl\| \int |\xi-\eta|^{-(d-1)}(|\xi|^s - |\eta|^s)\hat u(\eta)\,d\eta\biggr\|_{L_\xi^2}.
\]
As $0\leq s\leq 1$, we can estimate this by
\[
\biggl\| \int |\xi-\eta|^{-(d-1-s)} |\hat u(\eta)|\,d\eta\biggr\|_{L_\xi^2} = \|\, |\xi|^{-(d-1-s)}\ast |\hat u|\,\|_{L_\xi^2}. 
\]
Using the Lorentz-space\footnote{Here $L^{q,\alpha}$ denotes the Lorentz-space defined via the quasi-norm
\[
\|f\|_{L^{q,\alpha}} = \bigl\| \lambda |\{|f|>\lambda\} |^{\frac{1}{q}} \bigr\|_{L^\alpha((0,\infty),\frac{d\lambda}{\lambda})}. 
\]} 
refinements of Young's inequality and H\"older's inequality  (cf. \cite{Hunt, Oneil}), we may estimate
\begin{align*}
\|\, |\xi|^{-(d-1-s)}\ast |\hat u|\,\|_{L_\xi^2} & \lesssim \| |\xi|^{-(d-1-s)}\|_{L^{\frac{d}{d-1-s},\infty}} \| \hat u\|_{L^{\frac{2d}{d+2+2s},2}} \\
& \lesssim \| |\xi|^{-(s+1)}\|_{L^{\frac{d}{s+1},\infty}} \| |\xi|^{s+1}\hat u\|_{L^2}  \lesssim \| |\nabla|^{s+1} u\|_{L^2},
\end{align*}
which is acceptable.  In the estimates above, we have used $s\leq \tfrac{d}{2}-1$ to guarantee that the exponents above fall into acceptable intervals. \end{proof}

\subsection*{Acknowledgements} Y. Zhang was supported by the US National Science Foundation under grant number DMS-1620465.  We are grateful to M. Beceanu for explanations related to the notion of incoming/outgoing waves for radial wave equations.  We would also like to thank J. Colliander for introducing us to the related work of Strauss and Vazquez \cite{StraussVazquez}, and to an anonymous referee for pointing out the work of \cite{DonSch}. 


\end{document}